\documentclass[journal]{IEEEtran}%
\usepackage{amsfonts}
\usepackage{amsmath}
\usepackage{amssymb}
\usepackage{graphicx}
\usepackage[compress]{cite}%
\newtheorem{theorem}{Theorem}

\newtheorem{case}{Case}

\newtheorem{definition}[theorem]{Definition}
\newtheorem{example}[theorem]{Example}

\newtheorem{lemma}[theorem]{Lemma}

\newtheorem{proposition}[theorem]{Proposition}

\begin{document}

\title{On Bahadur Efficiency of Power Divergence Statistics}
\author{Peter Harremo\"{e}s,~\IEEEmembership{Member,~IEEE,} and~Igor
Vajda,~\IEEEmembership{Fellow,~IEEE}\thanks{Manuscript submitted February
2010.}\thanks{P. Harremo\"{e}s is with Copenhagen Business College,
Copenhagen, Denmark. I. Vajda is with Institute of Information Theory and
Automation, Prague, Czech Republic.}}
\maketitle

\begin{abstract}
It is proved that the information divergence statistic is infinitely more
Bahadur efficient than the power divergence statistics of the orders
$\alpha>1$ as long as the sequence of alternatives is contiguous with respect
to the sequence of null-hypotheses and the the number of observations per bin
increases to infinity is not very slow. This improves the former result in
Harremo\"{e}s and Vajda (2008) where the the sequence of null-hypotheses was
assumed to be uniform and the restrictions on on the numbers of observations
per bin were sharper. Moreover, this paper evaluates also the Bahadur
efficiency of the power divergence statistics of the remaining positive orders
$0\,<\alpha\leq1$. The statistics of these orders are mutually
Bahadur-comparable and all of them are more Bahadur efficient than the
statistics of the orders $\alpha>1.$ A detailed discussion of the technical
definitions and conditions is given, some unclear points are resolved, and the
results are illustrated by examples.

\end{abstract}


\section{Introduction}

\PARstart{P}{roblems} of detection, classification and identification are
often solved by the method of testing statistical hypotheses. Consider signals
$Y_{1},Y_{2},...,Y_{n}$\ collected \ from a random source independently at
time instants $i=1,2,...,n.$ Signal processing usually requires digitalization
based on appropriate quantization. Quantization of the signal space
$\mathcal{Y}$\ into $k$\ disjoint cells (or bins) $\mathcal{Y}_{n1}%
,\mathcal{Y}_{n2},...,\mathcal{Y}_{nk}$ reduces the signals $Y_{1}%
,Y_{2},...,Y_{n}$\ into simple $k$-valued indicators $I_{n}(Y_{1}),I_{n}%
(Y_{2}),...,I_{n}(Y_{n})$\ of their cover cells. Various hypotheses about the
data source represented by probability measures $Q_{n}$\ on $\mathcal{Y}$\ are
transformed by the quantization into discrete probability distributions
\[
Q_{n}=\left(  q_{n1}=Q(\mathcal{Y}_{n1}),...,q_{nk}=Q(\mathcal{Y}%
_{nk})\right)
\]
on the quantization cells where for no quantization cell $q_{nj}=0.$ These
hypothetical distributions need not be the same as the true distributions
$P_{n}=(p_{n1}=P(\mathcal{Y}_{n1}),...,p_{nk}=P(\mathcal{Y}_{nk}))$. The
latter distributions are usually unknown but, by the law of large numbers,
they can be approximated by the empirical distributions (vectors of relative
cell frequencies)%
\begin{equation}
\hat{P}_{n}=\left(  \hat{p}_{n1}=\frac{X_{n1}}{n},...,p_{nk}=\frac{X_{nk}}%
{n}\right)  =\frac{\mathbf{X}_{n}}{n}\label{empi}%
\end{equation}
where $X_{nj}$ is the numbers of the signals $Y_{1},Y_{2},...,Y_{n}$ in
$\mathcal{Y}_{nj}.$ Formally,
\begin{equation}
X_{nj}=%
{\textstyle\sum\limits_{i=1}^{n}}
1_{\left\{  Y_{i}\in\mathcal{Y}_{nj}\right\}  }=%
{\textstyle\sum\limits_{i=1}^{n}}
1_{\left\{  I_{n}(Y_{i})=j\right\}  },\text{ \ \ }1\leq j\leq k\label{af}%
\end{equation}
where $1_{A}$\ denotes the indicator of the event $A$. The problem is to
decide whether the signals $Y_{1},Y_{2},...,Y_{n}$\ are generated by the
source $(\mathcal{Y},Q)$ on the basis of the distributions $\hat{P}_{n},Q_{n}%
$. A classical method for solving this problem is the method of testing
statistical hypotheses in the spirit of Fisher, Neyman and Pearson. In our
case the hypothesis is
\begin{equation}
{\mathcal{H}}:P_{n}=Q_{n}\label{2}%
\end{equation}
and the decision is based either on the \emph{likelihood ratio statistic }%
\begin{equation}
\hat{T}_{1,n}=2%
{\textstyle\sum\limits_{j=1}^{k}}
X_{nj}\ln\frac{X_{nj}}{nq_{nj}}\label{lik}%
\end{equation}
or the \emph{Pearson }$\chi^{2}$\emph{-statistic}%
\begin{equation}
\hat{T}_{2,n}=%
{\textstyle\sum\limits_{j=1}^{k}}
\frac{(X_{nj}-nq_{nj})^{2}}{nq_{nj}}\label{Pears}%
\end{equation}
in the sense that the hypothesis is rejected when the statistic is large,
where "large" depends on the required decision error or risk \cite{Lehman2005}%
. $\medskip$

It is easy to see (c.f. (\ref{a}), (\ref{b}) below) that the classical test
statistics (\ref{lik}), (\ref{Pears}) are of the form%
\begin{equation}
\hat{T}_{\alpha,n}=2n\hat{D}_{\alpha,n}\overset{\text{def}}{=}2nD_{\alpha
}\left(  \hat{P}_{n},Q_{n}\right)  ,\quad\alpha\in\{1,2\}\label{clas}%
\end{equation}
where $D_{\alpha}\left(  P,Q\right)  $ for arbitrary $\alpha>0$\ and
distributions $P=(p_{1},...,p_{k})$, $Q=(q_{1},...,q_{k})$\ denotes the
divergence $D_{\phi_{\alpha}}\left(  P,Q\right)  $ of Csisz\'{a}r
\cite{Csiszar1963}$\ $for the power function%
\begin{equation}
\phi_{\alpha}(t)={\frac{t^{\alpha}-\alpha(t-1)-1}{\alpha(\alpha-1)}}\text{
\ when \ \ }\alpha\neq1\label{3b}%
\end{equation}
and
\begin{equation}
\phi_{1}(t)=\lim_{\alpha\rightarrow1}\phi_{\alpha}(t)=t\ln t-t+1.\label{3c}%
\end{equation}
The power divergences
\begin{equation}
D_{\alpha}\left(  P,Q\right)  =\frac{1}{\alpha(\alpha-1)}\left(
{\textstyle\sum\limits_{j=1}^{k}}
p_{j}^{\alpha}q_{j}^{1-\alpha}-1\right)  \text{ \ \ }\alpha\neq1\label{3a}%
\end{equation}
or the one-one related R\'{e}nyi divergences \cite{Renyi1961}%
\begin{equation}
D_{\alpha}\left(  P\Vert Q\right)  =\frac{1}{\alpha-1}\ln%
{\textstyle\sum\limits_{j=1}^{k}}
p_{j}^{\alpha}q_{j}^{1-\alpha}\text{ \ \ }\alpha\neq1\label{Re}%
\end{equation}
with the common information divergence limit
\begin{equation}
D_{1}\left(  P,Q\right)  =D_{1}\left(  P\Vert Q\right)  =%
{\textstyle\sum\limits_{j=1}^{k}}
p_{j}\ln\frac{p_{j}}{q_{j}}\label{3aa}%
\end{equation}
are often applied in various areas of information theory. In the present
context of detection and identification one can mention e.g. the work of
Kailath \cite{Kailath1967} who used the Bhattacharryya distance%
\[
B\left(  P,Q\right)  =-\ln%
{\textstyle\sum\limits_{j=1}^{k}}
\left(  p_{j}q_{j}\right)  ^{1/2}=\frac{1}{2}D_{1/2}\left(  P\Vert Q\right)
\]
which is one-one related to the Hellinger divergence. $\medskip$

In practical applications it is important to use the statistic $\hat
{D}_{\alpha_{\text{opt}},n}$ which is optimal in a sufficiently wide class of
divergence statistics $\hat{D}_{\alpha,n}$containing the standard statistical
proposals $\hat{D}_{1,n}$ and $\hat{D}_{2,n}$ appearing in (\ref{clas}). \ We
addressed this problem previously \cite{Harremoes2008, Harremoes2008d,
Harremoes2008e}. Our solution confirmed the classical statistical result of
Quine and Robinson \cite{Quine1985} who proved that the likelihood ratio
statistic $\hat{D}_{1,n}$\ is more efficient in the Bahadur sense than the
$\chi^{2}$-statistic $\hat{D}_{2,n}$ and extended the results of Beirlant et
al. \cite{Beirlant2001} and Gy\"{o}rfi et al. \cite{Gyorfi2000} dealing with
Bahadur efficiency of several selected power divergence \ statistics. Namely,
we evaluated the Bahadur efficiencies of the statistics $\hat{D}_{n\alpha}$ in
the domain $\alpha\geq1$ for the numbers $k=k_{n}$\ of quantization cells
slowly increasing with $n$ when the hypothetical distributions $Q_{n}$ are
uniform and the alternative distributions $P_{n}$\ are \textit{contiguous} in
the sense that $\lim_{n\rightarrow\infty}D_{\alpha}\left(  P_{n},Q_{n}\right)
$\ exists and \textit{identifiable} in the sense that this limit is positive.
We found that the Bahadur efficiencies decrease with the power parameter in
the whole domain $\alpha\geq1$. In the present paper we sharpen this result by
relaxing conditions on the rate of $k_{n}$\ and extend it considerably by
admitting non-uniform hypothetical distributions $Q_{n}$ and by evaluating the
Bahadur efficiencies also in the domain $0<\alpha\leq1.$

\section{Basic model}

Let $M(k)$ denote the set of all probability distributions $P=(p_{j}:1\leq
j\leq k)$ and%
\[
M(k|n)=\left\{  P\in M(k):nP\in\{0,1,\ldots\}^{k}\right\}
\]
its subset called the set of types in information theory. We consider
hypothetical distributions $Q_{n}=(q_{nj}:1\leq j\leq k)\in M(k)$\ restricted
by the condition $q_{nj}>0$\ and arbitrary alternative distributions
$P_{n}=(p_{nj}:1\leq j\leq k)\in M(k).$ The $\{0,1,\ldots\}^{k}$-valued
frequency counts $\mathbf{X}_{n}$ with coordinates introduced in (\ref{af})
are \textit{multinomially distributed} in the sense
\begin{equation}
\mathbf{X}_{n}\sim\mbox{Mult}_{k}(n,P_{n}),n=1,2,\ldots. \label{1}%
\end{equation}
Important components of the model are the empirical distributions $\widehat
{P}_{n}\in M(k|n)$\ defined by (\ref{empi}). Finally, for arbitrary $P\in
M(k)$\ and arbitrary $Q\in M(k)$\ with positive coordinates we consider the
power divergences (\ref{3a})-(\ref{3c}). For their properties we refer to
\cite{Liese1987, Liese2006, Read1988}. In particular, for the empirical and
hypothetical distributions $\hat{P}_{n},Q_{n}$\ we consider the power
divergence statistics $\widehat{D}_{\alpha,n}=D_{\alpha}\left(  \hat{P}%
_{n},Q_{n}\right)  $ (c.f. (\ref{clas}))defined by (\ref{3a}), (\ref{3aa}) for
all $\alpha>0$.

\begin{example}
\label{Example1 copy(1)}For $\alpha=2,$ $\alpha=1$\ and $\alpha=1/2$\ we get
the special power divergence statistics%
\begin{align}
\widehat{D}_{2,n}  &  ={\frac{1}{2}}\sum_{j=1}^{n}{\frac{(\widehat{p}%
_{nj}-q_{nj})^{2}}{q_{nj}}}={\frac{1}{2n}}\hat{T}_{2,n},\medskip\label{a}\\
\widehat{D}_{1,n}  &  =\sum_{j=1}^{n}\widehat{p}_{nj}\ln{\frac{\widehat
{p}_{nj}}{q_{nj}}}={\frac{1}{2n}}\hat{T}_{1,n},\medskip\label{b}\\
\widehat{D}_{1/2,n}  &  =2\sum_{j=1}^{n}\left(  \widehat{p}_{nj}^{1/2}%
-q_{nj}^{1/2}\right)  ^{2} \label{c}%
\end{align}
For testing the hypothesis $\mathcal{H}$\ of (\ref{2}) are usually used the
re-scaled versions
\begin{equation}
\widehat{T}_{\alpha,n}=2n\widehat{D}_{\alpha,n} \label{4}%
\end{equation}
distributed under $\mathcal{H}$\ asymptotically $\chi^{2}$ with $k-1$\ degrees
of freedom if $k$ is constant and asymptotically normally if $k=k_{n}$\ slowly
increases to infinity \cite[and references therein]{Morris1975, Gyorfi2002} .
The statistics (\ref{a}) and (\ref{b}) rescaled in this manner were already
mentioned in (\ref{Pears}) and (\ref{lik}). In (\ref{c}) is the Hellinger
divergence statistics rescaled by $2n$\ is known as \emph{Freeman--Tukey
statistic}
\begin{equation}
\widehat{T}_{1/2,n}=2n\widehat{D}_{1/2,n}=4\sum_{j=1}^{k}(\left(
X_{nj}\right)  ^{1/2}-\left(  nq_{nj}\right)  ^{1/2})^{2}. \label{FT}%
\end{equation}
\newline
\end{example}

\paragraph{Convention}

Unless the hypothesis ${\mathcal{H}}$\ is explicitly assumed, the random
variables, convergences and asymptotic relations are considered under the
alternative ${\mathcal{A}}$. Further, unless otherwise explicitly stated, the
asymptotic relations are considered for $n\longrightarrow\infty$ and the
symbols of the type%
\[
s_{n}\longrightarrow s\text{ \ \ and \ \ }s_{n}(\mathbf{X}_{n})\overset
{p}{\longrightarrow}s
\]
denote the ordinary numerical convergence and the stochastic convergence in
probability for $n\longrightarrow\infty$.$\medskip$

In this paper we consider the following assumptions.

$\medskip$

\begin{description}
\item[\textbf{A1:}] The number of cells $k=k_{n}\leq n$ of the distributions
from $M(k),\,M(k|n)$ depends on the sample size $n$ and increases to infinity.
In the rest of the paper the subscript $n$ is suppressed in the symbols
containing~$k$.

\item[\textbf{A2:}] The hypothetical distributions $Q_{n}=(q_{nj}>0:1\leq
j\leq k)$\ are regular in the sense that $\max_{j}q_{nj}\rightarrow0$ for
$n\rightarrow\infty$ and that there exists $\varrho>0$ such that%
\begin{equation}
q_{nj}>\frac{\varrho}{k}\text{ \ for all }1\leq j\leq k\text{\ and
}n=1,2,\ldots\text{ .} \label{5}%
\end{equation}

\item[\textbf{A3}$\alpha$\textbf{:}] The alternative ${\mathcal{A}}%
:(P_{n}:n=1,2,\ldots)$ \ is identifiable in the sense that there exits
$0<\Delta_{\alpha}<\infty$ such that
\begin{equation}
D_{\alpha,n}\overset{def}{=}D_{\alpha}(P_{n},Q_{n})\longrightarrow
\Delta_{\alpha}\text{\ \ under }{\mathcal{A}}. \label{6}%
\end{equation}

\end{description}

Under \textbf{A2}%
\begin{equation}
-\ln q_{nj}<\ln\frac{k}{\varrho}\text{ \ \ and \ \ }\ln^{2}q_{nj}<\ln^{2}%
\frac{k}{\varrho}. \label{qbounds}%
\end{equation}
Further, logical complement to the hypothesis ${\mathcal{H}}$\ is the
alternative denoted by ${\mathcal{A}}.$ By (\ref{2}), under ${\mathcal{A}}%
$\ the alternative distributions $P_{n}$\ differ from $Q_{n}$. Assumption
\textbf{A3}$\alpha$\ means that the alternative distributions are neither too
close to nor too distant from $Q_{n}$\ in the sense of $D_{\alpha}$-divergence
for given $\alpha>0$. Since for all $n=1,2,\ldots$
\[
D_{\alpha,n}=D_{\alpha}(Q_{n},Q_{n})\equiv0\ \ \text{so\ that }\Delta_{\alpha
}=0\text{\ \ under }{\mathcal{H}}%
\]
it is clear that the hypothesis${\mathcal{A}}$\ is under \textbf{A1},
\textbf{A2,}\thinspace\textbf{A3}$\alpha$\ distinguished from\ the hypothesis
${\mathcal{H}}$ by achieving a positive \mbox{$D_{\alpha }$-divergence} limit
$\Delta_{\alpha}$. In what follows we use the abbreviated notations
\begin{align}
\text{\textbf{A}}(\alpha)  &  =\left\{  \text{\textbf{A1},\textbf{A2}%
,\thinspace\textbf{A3}}\alpha\right\}  ,\label{6a}\\
\text{\textbf{A}}(\alpha_{1},\alpha_{2})  &  =\left\{  \text{\textbf{A1}%
,\textbf{A2},\thinspace\textbf{A3}}\alpha_{1}\text{, \textbf{A3}}\alpha
_{2}\right\}
\end{align}
for the combinations of assumptions.$\medskip$

\begin{definition}
\label{Def1}Under \textbf{A}($\alpha$) we say that the statistic $\widehat
{D}_{\alpha,n}$ is \emph{consistent} with parameter $\Delta_{\alpha}%
$\ appearing in \textbf{(}\ref{6}\textbf{) }if
\begin{equation}
\widehat{D}_{\alpha,n}\overset{p}{\longrightarrow}\Delta_{\alpha}\text{
\ \ under }{\mathcal{A}} \label{7}%
\end{equation}
and%
\begin{equation}
\widehat{D}_{\alpha,n}\overset{p}{\longrightarrow}0\text{ \ \ under
}{\mathcal{H}} \label{8}%
\end{equation}
i.e. if $\widehat{D}_{\alpha,n}\overset{p}{\longrightarrow}\Delta_{\alpha}%
$\ under both ${\mathcal{A}}$\ and ${\mathcal{H}}$. If (\ref{8}) is replaced
by the stronger condition that the\ expectation \textsf{E}$\widehat{D}%
_{\alpha,n}$\ tends\ to zero under ${\mathcal{H}}$, in symbols
\begin{equation}
\text{\textsf{E}}\left[  \left.  \widehat{D}_{\alpha,n}\right\vert
{\mathcal{H}}\right]  \longrightarrow0, \label{9}%
\end{equation}
then $\widehat{D}_{\alpha,n}$\ is said \emph{strongly consistent}.$\medskip$
\end{definition}

\begin{definition}
\label{Def2}We say that the statistic $\widehat{D}_{\alpha,n}$ is
\emph{Bahadur stable} if there is a continuous function with a Bahadur
relative function $\varrho_{\alpha}:$ $]0,\infty\lbrack^{2}$ $\rightarrow$
$]0,\infty\lbrack$ such that the probability of error function%
\begin{equation}
\mathsf{e}_{\alpha,n}(\Delta)=\mathsf{P}\left(  \left.  \widehat{D}_{\alpha
,n}>\Delta\right\vert {\mathcal{H}}\right)  ,\text{ \ \ }\Delta>0 \label{10}%
\end{equation}
corresponding to the test rejecting ${\mathcal{H}}$\ when $\widehat{D}%
_{\alpha,n}>\Delta$\ satisfies for all $\Delta_{1},\Delta_{2}>0$ the relation
\[
\frac{\ln\mathsf{e}_{\alpha,n}(\Delta_{1})}{\ln\mathsf{e}_{\alpha,n}%
(\Delta_{2})}\longrightarrow\varrho_{\alpha}\left(  \Delta_{1},\Delta
_{2}\right)  .
\]
If this condition holds then $\varrho_{\alpha}$\ is called the \emph{Bahadur
relative function}.
\end{definition}

Obviously, the Bahadur relative functions are multiplicative in the sense
\[
\varrho_{\alpha}\left(  \Delta_{1},\Delta_{2}\right)  \varrho_{\alpha}\left(
\Delta_{2},\Delta_{3}\right)  =\varrho_{\alpha}\left(  \Delta_{1},\Delta
_{3}\right)  .
\]
Statistics that are Bahadur stable have the nice property that the asymptotic
behavior of the error function $\mathsf{e}_{\alpha,n}(\Delta)$ is determined
by its behavior for just a single argument $\Delta^{\ast}>0.$ Indeed, if
$\widehat{D}_{\alpha,n}$ is Bahadur stable and if we define for a fixed
$\Delta^{\ast}>0$ the sequence%
\begin{equation}
c_{\alpha}^{\ast}(n)=-\frac{n}{\ln\mathsf{e}_{\alpha,n}(\Delta^{\ast})}
\label{11}%
\end{equation}
then for all $\Delta>0$%
\[
-{\frac{c_{\alpha}^{\ast}(n)}{n}}\ln\mathsf{e}_{\alpha,n}(\Delta
)\longrightarrow\varrho_{\alpha}\left(  \Delta,\Delta^{\ast}\right)  \text{
\ \ for all }\Delta>0.
\]
Moreover, if the expressions $\ -{c_{\alpha}(n)/n}\ln\mathsf{e}_{\alpha
,n}(\Delta)$\ \ converge for a sequence $c_{\alpha}\left(  n\right)  $\ then
the ratio $c_{\alpha}(n)/c_{\alpha}^{\ast}(n)$ tends to a constant.$\medskip$

\paragraph{Motivation of the next definition}

Suppose that condition \textbf{A(}$\alpha_{1},\alpha_{2}$\textbf{)} holds and
denote for each $\alpha\in\{\alpha_{1},\alpha_{2}\}$\ and $n=1,2,\ldots$\ by
$\Delta_{\alpha}+\varepsilon_{\alpha,n}$ the critical value of the statistics
$\widehat{D}_{\alpha_{i},n}$\ leading to the rejection of $\mathcal{H}$ with a
fixed power $0<\mathsf{p}<1$. In other words, let
\[
\mathsf{p}=\text{\textsf{P}}\left(  \widehat{D}_{\alpha,n}>\Delta_{\alpha
}+\varepsilon_{\alpha,n}\right)  \text{ \ \ for all }n=1,2,\ldots\text{ }%
\]
where the sequence $\varepsilon_{\alpha,n}=\varepsilon_{\alpha,n}(\mathsf{p}%
)$\ depends on the fixed $\mathsf{p}$. Since the assumed consistency of
$\widehat{D}_{\alpha,n}$\ implies that $\varepsilon_{\alpha,n}$ tends to
zero,\ the corresponding error probabilities $\mathsf{e}_{\alpha,n}%
(\Delta_{\alpha}+\varepsilon_{\alpha,n})$\ $=\mathsf{P}\left(  \left.
\widehat{D}_{\alpha,n}>\Delta_{\alpha}+\varepsilon_{\alpha,n}\right\vert
\mathcal{H}\right)  $ can be approximated by $\mathsf{e}_{\alpha,n}%
(\Delta_{\alpha})$\ $=\mathsf{P}\left(  \left.  \widehat{D}_{\alpha,n}%
>\Delta_{\alpha}\right\vert \mathcal{H}\right)  .$ By (\ref{12}),%
\[
-{\frac{c_{\alpha}(n)}{n}}\ln\mathsf{e}_{\alpha,n}(\Delta_{\alpha
})\longrightarrow g_{\alpha}(\Delta_{\alpha}).
\]
Hence the error $\mathsf{e}_{\alpha_{1},n}(\Delta_{\alpha_{1}})$ of the
statistic $\widehat{D}_{\alpha_{1},n}$\ tends to zero with the same
exponential rate as $\mathsf{e}_{\alpha_{2},m_{n}}(\Delta_{\alpha_{2}}%
)$\ achieved by $\widehat{D}_{\alpha_{2},m_{n}}$ for possibly different sample
sizes $m_{n}\neq n$\ with the property $m_{n}\longrightarrow\infty$\ if the
corresponding error exponents%
\begin{equation}
g_{\alpha_{1}}(\Delta_{\alpha_{1}})\frac{n}{c_{\alpha_{1}}(n)}\text{ \ \ and
\ \ }g_{\alpha_{2}}(\Delta_{\alpha_{2}})\frac{m_{n}}{c_{\alpha_{2}}(m_{n})}
\label{14}%
\end{equation}
tend to infinity with the same rate in the sense
\begin{equation}
\frac{m_{n}}{c_{\alpha_{2}}(m_{n})}=\frac{g_{\alpha_{1}}(\Delta_{\alpha_{1}}%
)}{g_{\alpha_{2}}(\Delta_{\alpha_{2}})}.\frac{n}{c_{\alpha_{1}}(n)}\left(
1+o(1)\right)  . \label{15}%
\end{equation}
The sample sizes $m_{n}$ and $n$ needed by the statistics $\widehat{D}%
_{\alpha_{2},n}$ and $\widehat{D}_{\alpha_{1},n}$\ to achieve the same rate of
convergence of errors are thus mutually related by the formula%
\begin{equation}
{\frac{m_{n}}{n}=}\text{ }\frac{g_{\alpha_{1}}(\Delta_{\alpha_{1}})}%
{g_{\alpha_{2}}(\Delta_{\alpha_{2}})}.\frac{c_{\alpha_{2}}(m_{n})}%
{c_{\alpha_{1}}(n)}\left(  1+o(1)\right)  . \label{16}%
\end{equation}

Obviously, the statistic $\widehat{D}_{\alpha_{1},n}$\ is asymptotically less
or more efficient than $\widehat{D}_{\alpha_{2},n}$\ if the ratio $m_{n}%
/n$\ of sample sizes needed to achieve the same rate of convergence of errors
to zero tends to a constant larger or smaller than $1$. This motivates the
following definition which refers to the typical convergent situation
\begin{equation}
{\frac{c_{\alpha_{2}}(m_{n})}{c_{\alpha_{1}}(n)}\longrightarrow}\text{
}c_{\alpha_{2}/\alpha_{1}}\text{ \ \ for some }0\leq c_{\alpha_{2}/\alpha_{1}%
}\leq\infty.\medskip\label{17}%
\end{equation}

\begin{definition}
\label{Def 2+}If there is a continuous function
\[
g_{\alpha}:\left]  0,\infty\right[  \mathcal{\ }{\mathcal{\rightarrow}}\left]
0,\infty\right[
\]
and a sequence $c_{\alpha}(n)$ such that for all $x>0$ the error function%
\begin{equation}
\mathsf{e}_{\alpha,n}(x)=\mathsf{P}\left(  \left.  D_{\alpha,n}>x\right\vert
\mathcal{H}\right)  ,\text{ \ \ }x>0 \label{12a}%
\end{equation}
satisfies for all $x>0$ the relation
\begin{equation}
-{\frac{c_{\alpha}(n)}{n}}\ln\mathsf{e}_{\alpha,n}(x)\longrightarrow
g_{\alpha}(x) \label{12}%
\end{equation}
then $g_{\alpha}$ is called \textit{Bahadur function} of the statistic
$D_{\alpha,n}$ generated by $c_{\alpha}(n)$. If (\ref{12}) is replaced by the
condition%
\begin{equation}
-{\frac{c_{\alpha}(n)}{n}}\ln\mathsf{e}_{\alpha,n}(x+\varepsilon
_{n})\longrightarrow g_{\alpha}(x)\text{ \ \ for arbitrary }\varepsilon
_{n}\longrightarrow0 \label{13}%
\end{equation}
then the function $g_{\alpha}$\ is \textit{strongly Bahadur}.$\medskip$
\end{definition}

\begin{definition}
\label{Def3}Let us assume that \textbf{A(}$\alpha_{1},\alpha_{2}$\textbf{)}
holds and that for each $\alpha\in\{\alpha_{1},\alpha_{2}\}$ the statistic
$\widehat{D}_{n,\alpha}$ is consistent with parameter $\Delta_{\alpha}$ and
has a Bahadur function $g_{\alpha}$ generated by a sequence $c_{\alpha}(n)$
such that (\ref{17}) is satisfied. Then the \emph{Bahadur efficiency} of
$\widehat{D}_{\alpha_{1},n}$ with respect to $\widehat{D}_{\alpha_{2},n}$ is
the number from the interval $[0,\infty]$\ defined by the formula
\begin{equation}
\mbox{BE}\left(  \widehat{D}_{\alpha_{1},n}\,;\text{ }\widehat{D}_{\alpha
_{2},n}\right)  ={\frac{g_{\alpha_{1}}(\Delta_{\alpha_{1}})}{g_{\alpha_{2}%
}(\Delta_{\alpha_{2}})}}.c_{\alpha_{2}/\alpha_{1}}. \label{18}%
\end{equation}
$\medskip$
\end{definition}

Hereafter we shall consider also the slightly modified concept of Bahadur
efficiency.$\medskip$

\begin{definition}
\label{Def3+}Let in addition to the assumptions of Definition \ref{Def3}, the
statistics $\widehat{D}_{\alpha_{1},n},$ $\widehat{D}_{\alpha_{2},n}$\ be
strongly consistent and the functions $g_{\alpha_{1}},$ $g_{\alpha_{2}}%
$\ strongly Bahadur. Then the Bahadur efficiency (\ref{18}) is said to be
\textit{Bahadur efficiency\ in the strong sense}.$\medskip$
\end{definition}

\paragraph{Motivation of Definition \ref{Def3+}}

Let the assumptions of this definition hold then for each $\alpha\in
\{\alpha_{1},\alpha_{2}\},and$ $u>0$ the function
\[
L_{\alpha,n}(u)=\text{\textsf{P}}\left(  \left.  {\widehat{T}}_{\alpha
,n}-\mathsf{E}\left[  \left.  {\widehat{T}}_{\alpha,n}\right\vert
{\mathcal{H}}\right]  \geq u\right\vert {\mathcal{H}}\right)  ,\ \text{(cf.
\ref{10})}%
\]
denotes the level of the error of the statistic
\[
{\widehat{T}}_{\alpha,n}-\mathsf{E}\left[  \left.  {\widehat{T}}_{\alpha
,n}\right\vert {\mathcal{H}}\right]  \equiv2n\left(  \widehat{D}_{\alpha
,n}-\mathsf{E}\left[  \left.  \widehat{D}_{\alpha,n}\right\vert {\mathcal{H}%
}\right]  \right)  \
\]
for critical value $u>0$. By the assumed strong consistency of $\widehat
{D}_{\alpha,n},$
\[
\frac{\text{\textsf{E}}\left[  \left.  \widehat{T}_{\alpha,n}\right\vert
{\mathcal{H}}\right]  }{2n}\longrightarrow0\ \ \ \ \ \mbox{(cf.(\ref{9}))}.
\]
This means that the sequence $c_{\alpha}(n)$ generating the strongly Bahadur
$g_{\alpha}$\ satisfies for all $t>0$\ the relation%
\begin{equation}
-{\frac{c_{\alpha}(n)}{n}}\ln\text{\textsf{P}}\left(  \left.  {\widehat{T}%
}_{\alpha,n}\geq\text{\textsf{E}}\left[  \left.  {\widehat{T}}_{\alpha
,n}\right\vert {\mathcal{H}}\right]  +2nt\right\vert {\mathcal{H}}\right)
\longrightarrow g_{\alpha}(t)\text{ \ \ .} \tag{cf. (\ref{13})}%
\end{equation}
Consequently, by the argument of Quine and Robinson \cite[p. 732]{Quine1985},
\[
\lim\nolimits_{n}-{\frac{c_{\alpha}(n)}{n}}\ln L_{\alpha,n}({\widehat{T}%
}_{\alpha,n})\overset{p}{\longrightarrow}g_{\alpha}(\Delta_{\alpha}).
\]
Hence \cite{Quine1985}, the error level $L_{\alpha_{1},n}({\widehat{T}%
}_{\alpha_{1},n})$ of the statistic ${\widehat{T}}_{\alpha_{1},n}%
=2n\widehat{D}_{\alpha_{1},n}$ is asymptotically equivalent to the error level
$L_{\alpha_{2},m_{n}}({\widehat{T}}_{\alpha_{2},m_{n}})$ of the statistic
${\widehat{T}}_{\alpha_{2},m_{n}}=2m_{n}\widehat{D}_{\alpha_{2},m_{n}}%
$\ achieved by a sample size $m_{n}$\ if the comparability (\ref{15}) takes
place or, in other words, if the sample sizes $n$\ and $m_{n}$\ are mutually
related by (\ref{16}). In other words, the concept of Bahadur efficiency
introduced in this paper coincides under the stronger assumptions of
Definition~\ref{Def3+} with the Bahadur efficiency of Quine and Robinson
\cite{Quine1985}.$\medskip$

Harremo\"{e}s and Vajda \cite{Harremoes2008} assumed the same strong
consistency as in Definition \ref{Def3+} but introduced the Bahadur efficiency
by the slightly different formula
\begin{equation}
\mbox{BE}\left(  \widehat{D}_{\alpha_{1},n}\,;\text{ }\widehat{D}_{\alpha
_{2},n}\right)  ={\frac{g_{\alpha_{1}}(\Delta_{\alpha_{1}})}{g_{\alpha_{2}%
}(\Delta_{\alpha_{2}})}}.\bar{c}_{\alpha_{2}/\alpha_{1}} \label{19}%
\end{equation}
where\footnote{Due to a missprint, $\alpha_{1}$ and $\alpha_{2}$\ were
interchanged behind the limit in \cite[Eq. 30]{Harremoes2008}, but the formula
was used in the correct form (\ref{19}). In the Appendix we prove that the
conclusions made on the basis of the original formula (\ref{19}) hold
unchanged under the present precised formula (\ref{18}).}%
\begin{equation}
\bar{c}_{\alpha_{2}/\alpha_{1}}=\lim_{n\longrightarrow\infty}\frac
{c_{\alpha_{2}}(n)}{c_{\alpha_{1}}(n)}. \label{20}%
\end{equation}

\section{Consistency}

In this section we study the consistency of the class of power divergence
statistics $D_{\alpha}(\widehat{P}_{n},Q_{n}),$ $\alpha>0.$ In the domain
$\alpha<0$ this consistency was studied in the particular case of uniform
$Q$\ by Harremo\"{e}s and Vajda \cite{Harremoes2008d}.$\medskip$

\begin{theorem}
\label{Theorem0}Let distributions $Q_{n}\in M\left(  k\right)  $ satisfy the
assumption \textbf{A}($\alpha$). Assume that $f$ is uniformly continuous. Then
the statistic $D_{f}(\widehat{P}_{n},Q_{n})$ is strongly consistent provided
\begin{equation}
\frac{n}{k}\longrightarrow\infty.\label{32}%
\end{equation}

\end{theorem}

\begin{proof}
Under $\mathcal{H}$ we have $D_{f}(P_{n},Q_{n})=D_{f}(Q_{n},Q_{n})=0.$ Hence
it suffices to prove
\begin{equation}
\left\vert \Lambda_{\alpha,n}\right\vert \overset{p}{\longrightarrow}0\text{
\ \ under both }\mathcal{H}\ \text{and }\mathcal{A}\label{35}%
\end{equation}
for $\Lambda_{\alpha,n}=D_{f}(\widehat{P}_{n},Q_{n})-D_{f}(P_{n},Q_{n})$. For
simplicity we skip the subscript $n$\ in the symbols $\widehat{P}_{n},P_{n},$
and $Q_{n}$, i.e. we substitute$\medskip$%
\begin{equation}
\widehat{P}_{n}=\widehat{P}=(\widehat{p}_{j}:1\leq j\leq k),\text{ \ \ }%
P_{n}=P=(p_{j}:1\leq j\leq k).\label{36}%
\end{equation}
This leads to the simplified formula $\Lambda_{\alpha,n}=D_{f}(\widehat
{P},Q)-D_{f}(P,Q).$ We can without loss of generality assume that $D_{f}(P,Q)$
is constant not only under $\mathcal{H}$\ (where the constant is automatically
$0$) but also under $\mathcal{A}$ (where the assumed detectability implies the
convergence $D_{f}(P,Q)\longrightarrow\Delta_{\alpha}$ for $0<\Delta_{\alpha
}<\infty$). In this asymptotic sense we use the equalities%
\begin{equation}
D_{f}(P,Q)=\frac{\sum q_{j}\left(  \frac{p_{j}}{q_{j}}\right)  ^{\alpha}%
-1}{\alpha\left(  \alpha-1\right)  }=\Delta_{\alpha}\label{37}%
\end{equation}
and%
\begin{equation}
\Lambda_{\alpha,n}=D_{f}(\widehat{P},U)-\Delta_{\alpha}\text{.}\label{38}%
\end{equation}
Choose some $0<s<1$ and define%
\[
f^{s}\left(  t\right)  =\left\{
\begin{array}
[c]{ll}%
f\left(  t\right)   & \text{for }t\geq s,\\
f\left(  s\right)  +f_{+}^{\prime}\left(  s\right)  \left(  t-s\right)   &
\text{for }0\leq t<s.
\end{array}
\right.
\]
Then
\[
0\leq f\left(  t\right)  -f^{s}\left(  t\right)  \leq f\left(  0\right)
-f^{s}\left(  0\right)
\]
so that (\ref{3a}) implies
\[
0\leq D_{f}\left(  P,Q\right)  -D_{f^{s}}\left(  P,Q\right)  \leq f\left(
0\right)  -f^{s}\left(  0\right)  .
\]
The function $f^{s}$ is Lipschitz with the Lipschitz constant $\lambda
=\max\left\{  \left\vert f_{+}^{\prime}\left(  s\right)  \right\vert
,\left\vert f^{\prime}\left(  \infty\right)  \right\vert \right\}  $ i.e.
$\left\vert f\left(  t_{1}\right)  -f\left(  t_{2}\right)  \right\vert
\leq\lambda\left\vert t_{1}-t_{2}\right\vert $ for all $t_{1},t_{2}\geq0.$
Then%
\begin{multline*}
\left\vert D_{f^{s}}(\widehat{P}_{n},Q)-D_{f^{s}}(P_{n},Q)\right\vert \\
=\left\vert \sum_{j=1}^{k}q_{j}\,f^{s}\left(  {\frac{\widehat{p}_{j}}{q_{j}}%
}\right)  -\sum_{j=1}^{k}\,q_{j}f^{s}\left(  {\frac{p_{j}}{q_{j}}}\right)
\right\vert \\
=\sum_{j=1}^{k}q_{j}\,\left\vert f^{s}\left(  {\frac{\widehat{p}_{j}}{q_{j}}%
}\right)  -f^{s}\left(  {\frac{p_{j}}{q_{j}}}\right)  \right\vert \leq
\sum_{j=1}^{k}q_{j}\lambda\,\left\vert {\frac{\widehat{p}_{j}}{q_{j}}}%
-{\frac{p_{j}}{q_{j}}}\right\vert \\
=\lambda\sum_{j=1}^{k}\,\left\vert \widehat{p}_{j}-p_{j}\right\vert
\leq\lambda\left(  \sum_{j=1}^{k}\frac{\left(  \widehat{p}_{j}-p_{j}\right)
^{2}}{p_{j}}\right)  ^{1/2}%
\end{multline*}
where in the last step we used the Schwarz inequality. Since
\begin{equation}
\mathsf{E}\left[  \left(  \widehat{p}_{j}-p_{j}\right)  ^{2}\right]
=p_{j}(1-p_{j})/n\leq\ p_{j}/n\label{45a}%
\end{equation}
it holds%
\begin{multline*}
\mathsf{E}\left\vert D_{f^{s}}(\widehat{P}_{n},Q)-D_{f^{s}}(P_{n}%
,Q)\right\vert \\
\leq\lambda\left(  \mathsf{E}\left[  \sum_{j=1}^{k}\frac{\left(  \widehat
{p}_{j}-p_{j}\right)  ^{2}}{p_{j}}\right]  \right)  ^{1/2}\leq\lambda\left(
\frac{k}{n}\right)  ^{1/2}.
\end{multline*}
Consequently,
\begin{multline*}
\mathsf{E}\left\vert D_{f}(\widehat{P}_{n},Q)-D_{f}(P_{n},Q)\right\vert \\
\leq2\left(  f\left(  0\right)  -f^{s}\left(  0\right)  \right)
+\lambda\left(  k/n\right)  ^{1/2}%
\end{multline*}
so that under (\ref{33})
\begin{multline*}
\underset{n\rightarrow\infty}{\lim\sup\text{ }}\mathsf{E}\left\vert
D_{f}(\widehat{P}_{n},Q_{n})-D_{f}(P_{n},Q_{n})\right\vert \\
\leq2\left(  f\left(  0\right)  -f^{s}\left(  0\right)  \right)  .
\end{multline*}
This holds for all $s>0.$ Since $f\left(  0\right)  -f^{s}\left(  0\right)
\longrightarrow0$ for $s\downarrow0,$ we see that in this case (\ref{32})
implies (\ref{35}).
\end{proof}

The interpretation of condition \ref{32} is that the mean number of
observations per bin should tend to infinity under $\mathcal{H}$. Note that
this condition does not exclude that we will observe empty cells.

Our results are concentrated in Theorem \ref{Theorem1} below. Its proof uses
the following auxiliary result.$\medskip$

\begin{lemma}
\label{Lemma1}For $x,y\geq0$ and $1\leq\alpha\leq2$ it holds%
\begin{equation}
L_{\alpha}\left(  x,y\right)  \leq\phi_{\alpha}(y)-\phi_{\alpha}(x)\leq
U_{\alpha}\left(  x,y\right)  \label{111}%
\end{equation}
\newline where%
\begin{equation}
L_{\alpha}\left(  x,y\right)  =(y-x)\phi_{\alpha}^{\prime}(x) \label{L}%
\end{equation}
and%
\begin{equation}
U_{\alpha}\left(  x,y\right)  =L_{\alpha}\left(  x,y\right)  +\frac{1}{\alpha
}x^{\alpha-2}\left(  y-x\right)  ^{2}.\medskip\label{1111}%
\end{equation}
\newline
\end{lemma}

\begin{proof}
First assume $1<\alpha<2.$ Since $\frac{1}{\alpha}x^{\alpha-2}\left(
y-x\right)  ^{2}$\ is nonnegative, it suffices to prove
\begin{equation}
\phi_{\alpha}(y)\geq\phi_{\alpha}(x)+\phi_{\alpha}^{\prime}(x)\left(
y-x\right)  \label{30ny}%
\end{equation}
and
\begin{equation}
\phi_{\alpha}(y)\leq\phi_{\alpha}(x)+\phi_{\alpha}^{\prime}(x)\left(
y-x\right)  +\frac{1}{\alpha}x^{\alpha-2}\left(  y-x\right)  ^{2}.
\label{31ny}%
\end{equation}
But Inequality (\ref{30ny}) is evident since the function $y\rightarrow
\phi_{\alpha}(y)$ is convex. We shall prove that the function%
\begin{multline*}
f\left(  y\right)  =\\
\phi_{\alpha}\left(  y\right)  -\left(  \phi_{\alpha}\left(  x\right)
+\phi_{\alpha}^{\prime}(x)\left(  y-x\right)  +\frac{1}{\alpha}x^{\alpha
-2}\left(  y-x\right)  ^{2}\right)
\end{multline*}
is non-positive. First we observe that $f(0)=f(x)=0$. By differentiating
$f\left(  y\right)  $\ we get%
\[
f^{\prime}\left(  y\right)  ={\frac{y^{\alpha-1}-1}{\alpha-1}-}\left(
\phi_{\alpha}^{\prime}(x)+\frac{2}{\alpha}x^{\alpha-2}\left(  y-x\right)
\right)
\]
so that $f^{\prime}\left(  x\right)  =0.$\ Differentiating once more we get%
\[
f^{\prime\prime}\left(  y\right)  =y^{\alpha-2}{-}\frac{2}{\alpha}x^{\alpha
-2}.
\]
Thus $f^{\prime\prime}(y)>0$ for $y<x_{\alpha}\overset{def}{=}\left(
\alpha/2\right)  ^{\frac{1}{2-\alpha}}x$ and $f^{\prime\prime}(y)<0$ for
$y>x_{\alpha}.$ Since $x_{\alpha}<x$ and $f(y)$ is concave on $[x_{\alpha}%
,1]$, it is maximized on this interval at $y=x$ where $f(x)=0$. Thus $f\left(
y\right)  \leq0$ on this interval and in particular $f(x_{\alpha})\leq0$. This
together with $f(0)=0$ and the convexity of $f\left(  y\right)  $\ on the
interval $[0,$ $x_{\alpha}]$ implies $f\left(  y\right)  \leq0$ for
$y\in\left[  0,x\right]  $. This completes the proof of the non-positivity of
$f\left(  y\right)  $, i.e. the proof of (\ref{31ny}). The cases $\alpha=2$
and $\alpha=1$\ follow by continuity.$\medskip$
\end{proof}

The main result of this section is the following theorem.$\medskip$

\begin{theorem}
\label{Theorem1}Let distributions $Q_{n}\in M\left(  k\right)  $ satisfy the
assumption \textbf{A}($\alpha$). Then the statistic $D_{\alpha}(\widehat
{P}_{n},Q_{n})$ is strongly consistent provided%
\begin{equation}
\rule{-7mm}{0mm}0<\alpha\leq2\text{ \ \ \ \ and \ \ \ \ }{\frac{n}{k}%
}\longrightarrow\infty\vspace*{-3mm} \label{33}%
\end{equation}
$\medskip$and consistent provided%
\begin{equation}
\rule{2mm}{0mm}\alpha>2\rule{10mm}{0mm}\text{\ \ \ and\ \ \ \ \ }{\frac
{n}{k\log k}}\longrightarrow\infty.\medskip\label{34}%
\end{equation}

\end{theorem}

\begin{proof}
$\medskip$We shall use the same notation as in the proof of Theorem
\ref{Theorem0}. In the proof we treat separately the
cases\renewcommand{\thecase}{\roman{case}}%
\begin{align*}
\mathbf{\ref{Caseii}}  &  :0<\alpha<1,\quad\text{ }\\
\mathbf{\ref{Caseiv}}  &  :1<\alpha\leq2,\quad\text{ }\\
\mathbf{\ref{Caseiii}}  &  :\alpha=1,\\
\mathbf{\ref{Casev}}  &  :\alpha>2.
\end{align*}

\begin{case}
[$0<\alpha<1$]\label{Caseii}This follows from Theorem \ref{Theorem0} because
$x\rightarrow\phi_{\alpha}\left(  x\right)  $ is uniformly continuous.
\end{case}

\begin{case}
[$1<\alpha\leq2$]\label{Caseiv}Here we get from (\ref{38})%
\begin{equation}
\Lambda_{\alpha,n}=\sum_{j=1}^{k}q_{j}\left(  \phi_{\alpha}\left(
\frac{\widehat{p}_{j}}{q_{j}}\right)  -\phi_{\alpha}\left(  \frac{p_{j}}%
{q_{j}}\right)  \right)  \label{lambda}%
\end{equation}
so that Lemma \ref{Lemma1} implies%
\begin{multline*}
\sum_{j=1}^{k}q_{j}L_{\alpha}\left(  \frac{p_{j}}{q_{j}},\frac{\widehat{p}%
_{j}}{q_{j}}\right)  \leq\Lambda_{\alpha,n}\leq\\
\\
\sum_{j=1}^{k}q_{j}L_{\alpha}\left(  \frac{p_{j}}{q_{j}},\frac{\widehat{p}%
_{j}}{q_{j}}\right)  +\sum_{j=1}^{k}q_{j}\frac{1}{\alpha}\left(  \frac{p_{j}%
}{q_{j}}\right)  ^{\alpha-2}\left(  \frac{\widehat{p}_{j}}{q_{j}}-\frac{p_{j}%
}{q_{j}}\right)  ^{2}%
\end{multline*}
and%
\[
\left\vert \Lambda_{\alpha,n}\right\vert \leq\left\vert \sum_{j=1}%
^{k}(\widehat{p}_{j}-p_{j})\phi_{\alpha}^{\prime}\left(  \frac{p_{j}}{q_{j}%
}\right)  \right\vert +\sum_{j=1}^{k}\frac{p_{j}^{\alpha-2}}{q_{j}^{\alpha-1}%
}\frac{\left(  \widehat{p}_{j}-p_{j}\right)  ^{2}}{\alpha}.
\]
We take the mean and get%
\begin{multline*}
E\left\vert \Lambda_{\alpha,n}\right\vert \leq\\
\mathsf{E}\left\vert \sum_{j=1}^{k}(\widehat{p}_{j}-p_{j})\phi_{\alpha
}^{\prime}\left(  \frac{p_{j}}{q_{j}}\right)  \right\vert +\sum_{j=1}^{k}%
\frac{p_{j}^{\alpha-2}}{\alpha q_{j}^{\alpha-1}}\mathsf{E}\left[  \left(
\widehat{p}_{j}-p_{j}\right)  ^{2}\right]
\end{multline*}
The terms on the right hand side are treated separately.%
\begin{align*}
\sum_{j=1}^{k}\frac{p_{j}^{\alpha-2}}{q_{j}^{\alpha-1}}\frac{\mathsf{E}\left[
\left(  \widehat{p}_{j}-p_{j}\right)  ^{2}\right]  }{\alpha}  &  =\sum
_{j=1}^{k}\frac{p_{j}^{\alpha-2}}{q_{j}^{\alpha-1}}\frac{\mathsf{E}\left[
\left(  n\widehat{p}_{j}-np_{j}\right)  ^{2}\right]  }{\alpha n^{2}}\\
&  =\sum_{j=1}^{k}\frac{p_{j}^{\alpha-2}}{q_{j}^{\alpha-1}}\frac{np_{j}\left(
1-p_{j}\right)  }{\alpha n^{2}}\\
&  \leq\frac{1}{\alpha n}\sum_{j=1}^{k}\frac{p_{j}^{\alpha-1}}{\left(
\frac{\rho}{k}\right)  ^{\alpha-1}}\\
&  \leq\frac{k^{\alpha-1}}{\alpha n\rho^{\alpha-1}}\sum_{j=1}^{k}p_{j}%
^{\alpha-1}.
\end{align*}
The function $P\rightarrow\sum_{j=1}^{k}p_{j}^{\alpha-1}$ is concave so it
attains its maximum for $P=\left(  1/k,1/k,\cdots,1/k\right)  .$ Therefore%
\begin{align*}
\sum_{j=1}^{k}\frac{p_{j}^{\alpha-2}}{q_{j}^{\alpha-1}}\frac{\mathsf{E}\left[
\left(  \widehat{p}_{j}-p_{j}\right)  ^{2}\right]  }{\alpha}  &  \leq
\frac{k^{\alpha-1}}{\alpha n\rho^{\alpha-1}}k\left(  \frac{1}{k}\right)
^{\alpha-1}\\
&  =\frac{1}{\alpha\rho^{\alpha-1}}\frac{k}{n}.
\end{align*}
Next we bound the first term.%
\begin{multline*}
\mathsf{E}\left\vert \sum_{j=1}^{k}(\widehat{p}_{j}-p_{j})\phi_{\alpha
}^{\prime}\left(  \frac{p_{j}}{q_{j}}\right)  \right\vert \\
\leq\mathsf{E}\left[  \left(  \sum_{j=1}^{k}(\widehat{p}_{j}-p_{j}%
)\phi_{\alpha}^{\prime}\left(  \frac{p_{j}}{q_{j}}\right)  \right)
^{2}\right]  ^{1/2}\\
=\left(  \sum_{i,j=1}^{k}Cov\left(  \widehat{p}_{i},\widehat{p}_{j}\right)
\phi_{\alpha}^{\prime}\left(  \frac{p_{i}}{q_{i}}\right)  \phi_{\alpha
}^{\prime}\left(  \frac{p_{j}}{q_{j}}\right)  \right)  ^{1/2}\\
=\frac{1}{n}\left(  \sum_{i,j=1}^{k}Cov\left(  n\widehat{p}_{i},n\widehat
{p}_{j}\right)  \phi_{\alpha}^{\prime}\left(  \frac{p_{i}}{q_{i}}\right)
\phi_{\alpha}^{\prime}\left(  \frac{p_{j}}{q_{j}}\right)  \right)  ^{1/2}\\
=\frac{1}{n}\left(
\begin{array}
[c]{c}%
\sum_{i=1}^{k}Var\left(  n\widehat{p}_{i}\right)  \left(  \phi_{\alpha
}^{\prime}\left(  \frac{p_{i}}{q_{i}}\right)  \right)  ^{2}\\
+\sum_{i\neq j}Cov\left(  n\widehat{p}_{i},n\widehat{p}_{j}\right)
\phi_{\alpha}^{\prime}\left(  \frac{p_{i}}{q_{i}}\right)  \phi_{\alpha
}^{\prime}\left(  \frac{p_{j}}{q_{j}}\right)
\end{array}
\right)  ^{1/2}.
\end{multline*}
This equals%
\begin{multline*}
\frac{1}{n}\left(
\begin{array}
[c]{c}%
\sum_{i=1}^{k}np_{i}\left(  1-p_{i}\right)  \left(  \phi_{\alpha}^{\prime
}\left(  \frac{p_{i}}{q_{i}}\right)  \right)  ^{2}\\
+\sum_{i\neq j}np_{i}p_{j}\phi_{\alpha}^{\prime}\left(  \frac{p_{i}}{q_{i}%
}\right)  \phi_{\alpha}^{\prime}\left(  \frac{p_{j}}{q_{j}}\right)
\end{array}
\right)  ^{1/2}\\
\leq\frac{1}{n^{1/2}}\left(
\begin{array}
[c]{c}%
\sum_{i=1}^{k}p_{i}\left(  \phi_{\alpha}^{\prime}\left(  \frac{p_{i}}{q_{i}%
}\right)  \right)  ^{2}\\
+\sum_{i,j}p_{i}p_{j}\phi_{\alpha}^{\prime}\left(  \frac{p_{i}}{q_{i}}\right)
\phi_{\alpha}^{\prime}\left(  \frac{p_{j}}{q_{j}}\right)
\end{array}
\right)  ^{1/2}.
\end{multline*}
This can be bounded as%
\begin{multline*}
\frac{1}{n^{1/2}}\left(
\begin{array}
[c]{c}%
\sum_{i=1}^{k}p_{i}\left(  \phi_{\alpha}^{\prime}\left(  \frac{p_{i}}{q_{i}%
}\right)  \right)  ^{2}\\
+\left(  \sum_{i}p_{i}\phi_{\alpha}^{\prime}\left(  \frac{p_{i}}{q_{i}%
}\right)  \right)  ^{2}%
\end{array}
\right)  ^{1/2}\\
\leq\left(  \frac{2}{n}\sum_{i=1}^{k}p_{i}\left(  \phi_{\alpha}^{\prime
}\left(  \frac{p_{i}}{q_{i}}\right)  \right)  ^{2}\right)  ^{1/2}.
\end{multline*}
These bounds can be combined into%
\begin{equation}
\mathsf{E}\left\vert \Lambda_{\alpha,n}\right\vert \leq\frac{1}{\alpha
\rho^{\alpha-1}}\frac{k}{n}+\left(  \frac{2}{n}\sum_{i=1}^{k}p_{i}\left(
\phi_{\alpha}^{\prime}\left(  \frac{p_{i}}{q_{i}}\right)  \right)
^{2}\right)  ^{1/2}. \label{faelles}%
\end{equation}
Under (\ref{33}) the first term tends to zero as $n\rightarrow\infty.$ The
last term does the same, which is seen from the inequalities%
\begin{multline*}
\frac{2}{n}\sum_{i=1}^{k}p_{i}\left(  \frac{\left(  \frac{p_{i}}{q_{i}%
}\right)  ^{\alpha-1}-1}{\alpha-1}\right)  ^{2}\\
\leq\frac{2}{n}\sum_{i=1}^{k}p_{i}\frac{\left(  \frac{p_{i}}{q_{i}}\right)
^{2\alpha-2}+1}{\left(  \alpha-1\right)  ^{2}}\\
=\frac{2}{n\left(  \alpha-1\right)  ^{2}}\left(  \sum_{i=1}^{k}q_{i}\left(
\frac{p_{i}}{q_{i}}\right)  ^{\alpha}\left(  \frac{p_{i}}{q_{i}}\right)
^{\alpha-1}+1\right) \\
\leq\frac{2}{n\left(  \alpha-1\right)  ^{2}}\sum_{i=1}^{k}q_{i}\left(
\frac{p_{i}}{q_{i}}\right)  ^{\alpha}\left(  \frac{1}{\rho/k}\right)
^{\alpha-1}+\frac{2}{n\left(  \alpha-1\right)  ^{2}}\\
=\frac{k^{\alpha-1}}{n}\frac{2}{\left(  \alpha-1\right)  ^{2}\rho^{\alpha-1}%
}\sum_{i=1}^{k}q_{i}\left(  \frac{p_{i}}{q_{i}}\right)  ^{\alpha}+\frac
{2}{n\left(  \alpha-1\right)  ^{2}}\\
=\frac{k^{\alpha-1}}{n}\frac{2\left(  \alpha\left(  \alpha-1\right)
\Delta+1\right)  }{\left(  \alpha-1\right)  ^{2}\rho^{\alpha-1}}+\frac
{2}{n\left(  \alpha-1\right)  ^{2}}.
\end{multline*}

\end{case}

\begin{case}
[$\alpha=1$]\label{Caseiii}For $\alpha=1$ in Inequality \ref{faelles} we get%
\[
\mathsf{E}\left\vert \Lambda_{1,n}\right\vert \leq\frac{k}{n}+\left(  \frac
{2}{n}\sum_{i=1}^{k}p_{i}\left(  \ln\frac{p_{i}}{q_{i}}\right)  ^{2}\right)
^{1/2}.
\]
Using $\ln p_{i}\leq0$ we find that last term on the right satisfies the
relations%
\begin{align}
&  \frac{1}{n}\sum_{j=1}^{k}p_{i}\left(  \ln^{2}p_{i}-2\ln p_{i}\ln q_{i}%
+\ln^{2}q_{i}\right) \label{sum}\\
&  =\frac{1}{n}\sum_{j=1}^{k}p_{i}\ln^{2}p_{i}-\frac{2}{n}\sum_{j=1}^{k}%
p_{i}\ln p_{i}\ln q_{i}+\frac{1}{n}\sum_{j=1}^{k}p_{i}\ln^{2}q_{i}\nonumber\\
&  \leq\frac{1}{n}\sum_{j=1}^{k}p_{i}\ln^{2}p_{i}+\frac{\ln^{2}\frac
{k}{\varrho}}{n}\sum_{j=1}^{k}p_{i}\leq\frac{1}{n}\sum_{j=1}^{k}p_{i}\ln
^{2}p_{i}+\frac{\ln^{2}\frac{k}{\varrho}}{n}.\nonumber
\end{align}
The function $x\rightarrow x\ln^{2}x$ is concave in the interval $\left[
0;e^{-1}\right]  $ and convex in the interval $\left[  e^{-1};1\right]  .$
Therefore we we can apply the method of \cite{Harremoes2001h} to verify that
$\sum_{i=1}^{k}p_{i}\ln^{2}p_{i}$ attains its maximum for a mixture of uniform
distributions on $k$ points and on subset of $k-1$ of these points. Thus%
\begin{multline}
\frac{1}{n}\sum_{i=1}^{k}p_{i}\ln^{2}p_{j}\leq\frac{1}{n}\sum_{i=1}^{k}%
\frac{1}{k-1}\ln^{2}\left(  \frac{1}{k}\right) \label{second}\\
=\frac{k\ln^{2}k}{n\left(  k-1\right)  }\leq\frac{2\ln^{2}k}{n}%
\end{multline}
and we can conclude that under (\ref{33}) the first term in (\ref{sum}) tends
to zero as $n$ tends to infinity. Obviously, under (\ref{33}) also the second
term in (\ref{lambda}) tends to zero so that the desired relation (\ref{35})
holds.\bigskip
\end{case}

\begin{case}
[$\alpha>2$]\label{Casev} By \textbf{A2},%
\begin{align*}
D_{\alpha}(P,Q)  &  =\frac{1}{\alpha(\alpha-1)}\left(  \sum_{j=1}^{k}%
p_{j}^{\alpha}q_{j}^{1-\alpha}-1\right) \\
&  \geq\frac{1}{\alpha(\alpha-1)}\left(  \left(  \frac{k}{\varrho}\right)
^{\alpha-1}\sum_{j=1}^{k}p_{j}^{\alpha}-1\right)
\end{align*}
so that
\begin{equation}
\sum_{j=1}^{k}p_{j}^{\alpha}\leq\left(  \alpha(\alpha-1)\Delta+1\right)
\left(  \frac{\varrho}{k}\right)  ^{\alpha-1} \label{pialpha}%
\end{equation}
where we replaced $D_{\alpha}(P,Q)$ by $\Delta=\Delta_{\alpha}$\ in the sense
of (\ref{37}). Further, by the Taylor formula%
\begin{equation}
\widehat{p}_{j}^{\alpha}=p_{j}^{\alpha}+\alpha\,p_{j}^{\alpha-1}(\widehat
{p}_{j}-p_{j})+{\frac{\alpha(\alpha-1)}{2}}\,\xi_{j}^{\alpha-2}(\widehat
{p}_{j}-p_{j})^{2} \label{50}%
\end{equation}
where $\xi_{j}$ is between $p_{j}$ and $\widehat{p}_{j}.$ We shall look for a
highly probable upper bound on $\widehat{p}_{j}$. Choose any $b>1$ and
consider\ the random event
\[
E_{nj}(b)=\{\widehat{p}_{j}\geq b\max\left\{  p_{j},q_{j}\right\}  \}.
\]
We shall prove that under (\ref{34}) it holds%
\begin{equation}
\boldsymbol{\pi}_{n}(b)\overset{def}{=}\mathsf{P}\left(  \cup_{j}%
E_{nj}(b)\right)  \longrightarrow0. \label{51}%
\end{equation}
The components $X_{j}=X_{nj}$ of the observation vector $\boldsymbol{X}_{n}%
$\ defined in Section 1 are approximately \textit{Poisson distributed},
Po$\left(  np_{j}\right)  ,$ so that
\begin{multline*}
\mathsf{P}\left(  \widehat{p}_{j}\geq b\max\left\{  p_{j},q_{j}\right\}
\right)  =\mathsf{P}\left(  X_{j}\geq nb\max\left\{  p_{j},q_{j}\right\}
\right) \\
\leq\exp\{-D_{1}\left(  \text{Po}\left(  b\max\left\{  np_{j},nq_{j}\right\}
\right)  ,\text{ Po}\left(  np_{j}\right)  \right)  \}
\end{multline*}
for the divergence $D_{1}(P,Q)$ defined by (\ref{3a})-(\ref{3c}) with
$P,Q$\ replaced by the corresponding Poisson distributions. But%
\begin{equation}
D_{1}\left(  \text{Po}\left(  bnq_{j}\right)  ,\text{ Po}\left(
np_{j}\right)  \right)  =np_{j}\phi_{1}\left(  b\right)  \label{Poisson}%
\end{equation}
for the logarithmic function $\phi_{1}\geq0$ introduced in (\ref{3b}).\ Since
for all $0\leq p_{j},q_{j}\leq1$%
\[
\phi_{1}\left(  \frac{b\max\left\{  p_{j},q_{j}\right\}  _{j}}{p_{j}}\right)
\geq\phi_{1}\left(  b\right)  >1\text{ \ for }b>1,\text{\ }%
\]
it holds\
\begin{multline*}
D_{1}\left(  \text{Po}\left(  b\max\left\{  np_{j},nq_{j}\right\}  \right)
,\text{ Po}\left(  np_{j}\right)  \right) \\
\geq D_{1}\left(  \text{Po}\left(  bnq_{j}\right)  ,\text{ Po}\left(
nq_{j}\right)  \right)  .
\end{multline*}
Consequently,
\begin{align}
\boldsymbol{\pi}_{n}(b)  &  \leq\sum_{j}\mathsf{P}\left(  \widehat{p}_{j}\geq
b\max\left\{  p_{j},q_{j}\right\}  \right) \nonumber\\
&  \leq\sum_{j}\exp\{-D_{1}\left(  \text{Po}\left(  b\max\left\{
np_{j},nq_{j}\right\}  \right)  ,\text{ Po}\left(  np_{j}\right)  \right)
\}\nonumber\\
&  \leq\sum_{j}\exp\left\{  -D_{1}\left(  \text{Po}\left(  bnq_{j}\right)
,\text{ Po}\left(  nq_{j}\right)  \right)  \right\} \nonumber\\
&  =\sum_{j}\exp\left\{  -nq_{j}\phi_{1}\left(  b\right)  \right\}  \text{
\ \ \ \ (cf. (\ref{Poisson}))}\nonumber\\
&  \leq k\exp\left\{  -n\frac{\varrho}{k}\phi_{1}\left(  b\right)  \right\}
=k^{1-\frac{n}{k\log k}\varrho\phi_{1}\left(  b\right)  }. \label{52}%
\end{align}
Assumption (\ref{34}) implies that the exponent in (\ref{52}) tends to
$-\infty$ so that (\ref{51}) holds. Therefore it suffices to prove (\ref{35})
under the condition that for all sufficiently large $n$\ the random events
$\cup_{j}E_{nj}(b)$ fail to take place, i.e. that%
\begin{equation}
\widehat{p}_{j}<b\max\left\{  p_{j},q_{j}\right\}  \text{ \ \ for all \ }1\leq
j\leq k. \label{53}%
\end{equation}
Let us start with the fact that under (\ref{53})\ it holds $\xi_{j}%
\leq\left\{  bp_{j},bq_{j}\right\}  $ and then%
\begin{equation}
\xi_{j}^{\alpha-2}\leq\left(  \max\left\{  bp_{j},bq_{j}\right\}  \right)
^{\alpha-2}\leq b^{\alpha-2}p_{j}^{\alpha-2}+b^{\alpha-2}\frac{\varrho
^{\alpha-2}}{k^{\alpha-2}}. \label{54}%
\end{equation}
Applying this in the Taylor formula (\ref{50}) we obtain%
\begin{align*}
\left\vert \widehat{p}_{j}^{\alpha}-p_{j}^{\alpha}\right\vert  &  \leq
\alpha\,p_{j}^{\alpha-1}\left\vert \widehat{p}_{j}-p_{j}\right\vert \\
&  +{\frac{\alpha(\alpha-1)b^{\alpha-2}}{2}}\,\left(  p_{j}^{\alpha-2}%
+\frac{\varrho^{\alpha-2}}{k^{\alpha-2}}\right)  (\widehat{p}_{j}-p_{j})^{2}.
\end{align*}
Hence under (\ref{53}) we get from (\ref{lambda}) and Lemma \ref{Lemma1}
\begin{multline}
\left\vert \Lambda_{\alpha,n}\right\vert \leq\frac{k^{\alpha-1}}{\alpha\left(
\alpha-1\right)  }\sum_{j=1}^{n}\alpha\,p_{j}^{\alpha-1}\left\vert \widehat
{p}_{j}-p_{j}\right\vert \nonumber\\
+\frac{k^{\alpha-1}}{\alpha\left(  \alpha-1\right)  }\sum_{j=1}^{n}%
{\frac{\alpha(\alpha-1)b^{\alpha-2}}{2}}\,\left(  p_{j}^{\alpha-2}%
+\frac{\varrho^{\alpha-2}}{k^{\alpha-2}}\right)  (\widehat{p}_{j}-p_{j})^{2}.
\end{multline}
Applying (\ref{pialpha}) and using Jensen's inequality and the expectation
bound (\ref{45a}), we upper bound $\mathbf{E}\left\vert \Lambda_{\alpha
,n}\right\vert $ by
\begin{align}
&  \frac{\left(  \alpha\left(  \alpha-1\right)  \Delta+1\right)  ^{1/2}%
}{\alpha\left(  \alpha-1\right)  }\left(  \frac{\sum p_{j}^{\alpha-1}%
}{k^{1-\alpha}n}\right)  ^{1/2}\nonumber\\
&  +{\frac{b^{\alpha-2}k^{\alpha-1}}{2}}\sum_{j=1}^{k}\left(  p_{j}^{\alpha
-2}+\frac{\varrho^{\alpha-2}}{k^{\alpha-2}}\right)  \mathsf{E}\left[
(\widehat{p}_{j}-p_{j})^{2}\right]  \medskip\nonumber\\
&  \leq\frac{\left(  \alpha\left(  \alpha-1\right)  \Delta+1\right)  ^{1/2}%
}{\alpha\left(  \alpha-1\right)  }\left(  \frac{\sum p_{j}^{\alpha-1}%
}{k^{1-\alpha}n}\right)  ^{1/2}\nonumber\\
&  +{\frac{b^{\alpha-2}k^{\alpha-1}}{2}}\sum_{j=1}^{k}\left(  p_{j}^{\alpha
-2}+\frac{\varrho^{\alpha-2}}{k^{\alpha-2}}\right)  \frac{p_{j}}{n}%
\medskip\nonumber\\
&  =\frac{\left(  \alpha\left(  \alpha-1\right)  \Delta+1\right)  ^{1/2}%
}{\alpha\left(  \alpha-1\right)  }\left(  \frac{\sum p_{j}^{\alpha-1}%
}{k^{1-\alpha}n}\right)  ^{1/2}\label{56}\\
&  +{\frac{b^{\alpha-2}}{2}}\frac{k^{\alpha-1}\sum_{j=1}^{k}p_{j}^{\alpha-1}%
}{n}+\frac{b^{\alpha-2}\varrho^{\alpha-2}}{2}\frac{k}{n}.\nonumber
\end{align}
Obviously, under (\ref{53})\ the desired relation (\ref{35}) holds if the
assumption (\ref{34}) implies the convergence%
\begin{equation}
\frac{\sum p_{j}^{\alpha-1}}{k^{1-\alpha}n}\rightarrow0. \label{convergence}%
\end{equation}
However, by Schwarz inequality and (\ref{pialpha}),
\begin{multline*}
\sum_{j=1}^{k}\,p_{j}^{\alpha-1}=\sum_{j=1}^{k}p_{j}\,\left(  p_{j}^{\alpha
-1}\right)  ^{(\alpha-2)/(\alpha-1)}\medskip\\
\leq\,\left(  \sum_{j=1}^{k}p_{j}p_{j}^{\alpha-1}\right)  ^{(\alpha
-2)/(\alpha-1)}=\left(  \sum_{j=1}^{k}p_{j}^{\alpha}\right)  ^{(\alpha
-2)/(\alpha-1)}\medskip\\
\leq\left(  \left(  \alpha\left(  \alpha-1\right)  \Delta+1\right)
\frac{\varrho^{\alpha-1}}{k^{\alpha-1}}\right)  ^{(\alpha-2)/(\alpha-1)}\\
=\frac{\varrho^{\alpha-2}\left(  \alpha\left(  \alpha-1\right)  \Delta
+1\right)  ^{(\alpha-2)/(\alpha-1)}}{k^{\alpha-2}}%
\end{multline*}
so that the validity of (\ref{35}) under (\ref{34}) is obvious and the proof
is complete.$\medskip$
\end{case}
\end{proof}

Condition \ref{34} is stronger than Condition \ref{32} and implies that for
any fixed number $a>0$ eventually any bins will contain more than $a$ observations.

\section{Bahadur efficiency}

In this section we study the Bahadur efficiency in the class of power
divergence statistics $\hat{D}_{\alpha,n}=D_{\alpha}(\hat{P}_{n},Q_{n})$,
$\alpha>0$. As before, we use the simplified notations%
\[
P_{n}=P,\text{ }Q_{n}=Q\text{ \ \ and \ }k_{n}=k.
\]
The results are concentrated in Theorem~\ref{Theorem2} below. Its proof is
based on the following lemmas. The first two of them make use of the R\'{e}nyi
divergences of orders $\alpha>0$%
\begin{align*}
D_{\alpha}\left(  P\Vert Q\right)   &  ={\frac{1}{\alpha-1}}\ln\sum_{j=1}%
^{k}p_{j}^{\alpha}q_{j}^{1-\alpha},\ \ \\
D_{1}\left(  P\Vert Q\right)   &  =\lim_{\alpha\rightarrow1}D_{\alpha}\left(
P\Vert Q\right)  =D\left(  P\Vert Q\right)
\end{align*}
where $D\left(  P\Vert Q\right)  $\ is the classical information divergence
denoted above by $D_{1}\left(  P,Q\right)  $. There is a monotone relationship
between the R\'{e}nyi and power divergences given by the formula%
\begin{align}
D_{\alpha}\left(  P\Vert Q\right)   &  ={\frac{1}{\alpha-1}}\ln\left(
1+\alpha\left(  \alpha-1\right)  D_{\alpha}\left(  P,Q\right)  \right)
,\text{ \ \ }\label{relation}\\
D_{1}\left(  P\Vert Q\right)   &  =D_{1}\left(  P,Q\right)  .
\end{align}

\begin{lemma}
Let $P$ and $Q$ be probability vectors on the set $\mathcal{X}$. If
$\alpha<\beta$ then%
\[
D_{\alpha}\left(  P\Vert Q\right)  \leq D_{\beta}\left(  P\Vert Q\right)  .
\]
with equality if and only there exists a subset $A\subseteq\mathcal{X}$ such
that $P=Q\left(  \cdot\mid A\right)  .\medskip$
\end{lemma}

\begin{proof}
By Jensen's inequality
\begin{align*}
D_{\alpha}\left(  P\Vert Q\right)   &  ={\frac{1}{\alpha-1}}\ln\sum_{j=1}%
^{k}p_{j}^{\alpha}q_{j}^{1-\alpha}\\
&  ={\frac{1}{\alpha-1}}\ln\sum_{j=1}^{k}p_{j}\left(  \left(  \frac{p_{j}%
}{q_{j}}\right)  ^{\beta-1}\right)  ^{\frac{\alpha-1}{\beta-1}}\\
&  \leq{\frac{1}{\alpha-1}}\ln\left(  \sum_{j=1}^{k}p_{j}\left(  \frac{p_{j}%
}{q_{j}}\right)  ^{\beta-1}\right)  ^{\frac{\alpha-1}{\beta-1}}\\
&  \leq{\frac{1}{\beta-1}}\ln\sum_{j=1}^{k}p_{j}\left(  \frac{p_{j}}{q_{j}%
}\right)  ^{\beta-1}\\
&  =D_{\beta}\left(  P\Vert Q\right)  .
\end{align*}
The equality takes place if and only if $\left(  \frac{p_{j}}{q_{j}}\right)
^{\beta-1}$ is constant $P$-almost surely. Therefore $\frac{p_{j}}{q_{j}}$ is
constant on the support of $P$ that we shall denote $A.$ Now $P$ equals $Q$
conditioned on $A.$
\end{proof}

\begin{lemma}
\label{Lemma2a}Let $0<\alpha\leq1.$ If
\begin{equation}
{\frac{\,n}{k\ln n}}\longrightarrow\infty. \label{62a}%
\end{equation}
and $q_{\max}\rightarrow0$ as $n\rightarrow\infty$ then the statistic
$\widehat{D}_{\alpha,n}$ is Bahadur stable and consistent and the constant
sequence generates the Bahadur function%
\begin{multline}
g_{\alpha}(\Delta)=\label{63}\\
\left\{
\begin{array}
[c]{ll}%
\displaystyle{\frac{\ln\left(  1+\alpha(\alpha-1)\,\Delta\right)  }{\alpha-1}%
},\quad\Delta>0 & \mbox{when}\ 0<\alpha<1\medskip\\
\displaystyle\lim_{\alpha\rightarrow1}g_{\alpha}(\Delta)=\Delta,\quad
\text{\ \ \ \ \ }\Delta>0 & \mbox{when}\ \alpha=1.\medskip
\end{array}
\right.  \medskip
\end{multline}

\end{lemma}

\begin{proof}
Let us first consider $0<\alpha<1$. The minimum of $D_{1}(P,Q)$ given
$D_{\alpha}(P\Vert Q)\geq\Delta$ is lower bounded by $\Delta.$ Let
$\varepsilon>0$ be given. If $q_{\max}$ is sufficiently small there exist sets
$A_{-}\subseteq A_{+}$ such that
\[
-\ln Q\left(  A_{+}\right)  \leq\Delta\leq-\ln Q\left(  A_{-}\right)
\leq\Delta+\varepsilon.
\]
Let $P_{s}$ denote the mixture $\left(  1-s\right)  Q\left(  \cdot\mid
A_{+}\right)  +sQ\left(  \cdot\mid A_{-}\right)  .$ Then $s\rightarrow
D_{\alpha}\left(  P_{s}\Vert Q\right)  $ is a continuous function satisfying%
\begin{align*}
D_{\alpha}\left(  P_{0}\Vert Q\right)   &  \leq\Delta,\\
D_{\alpha}\left(  P_{1}\Vert Q\right)   &  \geq\Delta.
\end{align*}
In particular there exist $s\in\left[  0,1\right]  $ such that $D_{\alpha
}\left(  P_{s}\Vert Q\right)  =\Delta.$ For this $s$ we have%
\begin{multline*}
D_{1}\left(  P_{s},Q\right) \\
\leq\left(  1-s\right)  D_{1}\left(  Q\left(  \cdot\mid A_{+}\right)
,Q\right)  +sD_{1}\left(  Q\left(  \cdot\mid A_{-}\right)  ,Q\right) \\
=\left(  1-s\right)  \left(  -\ln Q\left(  A_{+}\right)  \right)  +s\left(
-\ln Q\left(  A_{-}\right)  \right) \\
\leq\left(  1-s\right)  \Delta+s\left(  \Delta+\varepsilon\right)
=\Delta+\varepsilon.
\end{multline*}
Hence%
\[
\Delta\leq\inf D_{1}\left(  P,Q\right)  \leq\Delta+\varepsilon
\]
where the infimum is taken over all $P$ satisfying $D_{\alpha}\left(  P\Vert
Q\right)  =\Delta$ and where $n$ is sufficiently large. This holds for all
$\varepsilon>0$ so the Bahadur function of the statistic $D_{\alpha}\left(
\hat{P}\Vert Q\right)  $ is $g\left(  \Delta\right)  =\Delta.$ The Bahadur
function of the power divergence statistics $D_{\alpha}\left(  \hat
{P},Q\right)  $ can be calculated using Equality \ref{relation}.$\medskip$
\end{proof}

\begin{lemma}
\label{Lemma2}Let $\alpha>1.$ If assumptions \textbf{A}($\alpha$) holds for
for the uniform distributions $Q_{n}=U$\ and the sequence
\begin{equation}
c_{\alpha}(n)={\frac{\,k^{(\alpha-1)/\alpha}}{\ln k}} \label{60}%
\end{equation}
satisfies the condition%
\begin{equation}
{\frac{n}{c_{\alpha}(n)\,k\ln n}}\longrightarrow\infty\label{59}%
\end{equation}
then the statistic $\widehat{D}_{\alpha,n}=D_{\alpha}(\hat{P}_{n},Q_{n})$ is
consistent and the sequence (\ref{60}) generates the Bahadur function
\begin{equation}
g_{\alpha}(\Delta)=\left(  \alpha(\alpha-1)\,\Delta\right)  ^{1/\alpha}%
,\quad\Delta>0.\medskip\label{61}%
\end{equation}

\end{lemma}

\begin{proof}
If the sequence (\ref{60}) satisfies (\ref{59}) then Theorem 1 implies the
consistency of $\widehat{D}_{\alpha,n}$. Formula (\ref{61}) was already
mentioned in Example~2 above with a reference to Harremo\"{e}s and Vajda
\cite{Harremoes2008}).$\medskip$
\end{proof}

\begin{theorem}
\label{Theorem2}Let the assumption \textbf{A}($\alpha_{1},\alpha_{2}$) hold
where $0<\alpha_{1}<\alpha_{2}.$ If%
\begin{equation}
{\frac{k\ln n}{n}}\longrightarrow0 \label{67}%
\end{equation}
then the statistics
\begin{align*}
\widehat{D}_{\alpha_{1},n}  &  =D_{\alpha_{1}}(\hat{P}_{n},Q_{n}),\\
\widehat{D}_{\alpha_{2},n}  &  =D_{\alpha_{2}}(\hat{P}_{n},Q_{n})
\end{align*}
satisfy the relation\
\begin{multline}
\mbox{BE}\left(  \widehat{D}_{\alpha_{1},n};\widehat{D}_{\alpha_{2},n}\right)
\label{68}\\
=\left\{
\begin{array}
[c]{ll}%
\displaystyle{\frac{\alpha_{2}-1}{\alpha_{1}-1}}\cdot{\frac{\ln\left(
1+\alpha_{1}(\alpha_{1}-1)\Delta_{\alpha_{1}}\right)  \,}{\ln\left(
1+\alpha_{2}(\alpha_{2}-1)\,\Delta_{\alpha_{2}}\right)  }} &
\mbox{for}\ \alpha_{2}<1\medskip\medskip\medskip\\
\displaystyle{\frac{1}{\alpha_{1}-1}}\cdot{\frac{\ln\left(  1+\alpha
_{1}(\alpha_{1}-1)\Delta_{\alpha_{1}}\right)  \,}{\Delta_{\alpha_{2}}}} &
\mbox{for}\ \alpha_{2}=1.\medskip
\end{array}
\right.
\end{multline}
If
\begin{equation}
{\frac{k^{2-1/\alpha_{2}}\ln n}{n}}\longrightarrow0 \label{69}%
\end{equation}
then the statistics $\widehat{D}_{\alpha_{1},n}=D_{\alpha_{1}}(\hat{P}_{n},U)$
and $\widehat{D}_{\alpha_{2},n}=D_{\alpha_{2}}(\hat{P}_{n},U)$ satisfy the
relation\
\begin{equation}
\mbox{BE}\left(  \widehat{D}_{\alpha_{1},n};\widehat{D}_{\alpha_{2},n}\right)
=\infty\text{ \ \ \ }\mbox{for}\text{ }\alpha_{2}>1.\medskip\label{70}%
\end{equation}

\end{theorem}

\begin{proof}
By Lemma \ref{Lemma2a}, the assumptions of Definition~\ref{Def3} hold. The
first assertion follows directly from Definition~3 since, by Lemma~
\ref{Lemma2a},
\begin{multline}
\frac{g_{\alpha_{1}}(\Delta_{\alpha_{1}})}{g_{\alpha_{2}}(\Delta_{\alpha_{2}%
})}=\label{70a}\\
\left\{
\begin{array}
[c]{ll}%
\displaystyle{\frac{\alpha_{2}-1}{\alpha_{1}-1}}\cdot{\frac{\ln\left(
1+\alpha_{1}(\alpha_{1}-1)\Delta_{\alpha_{1}}\right)  }{\ln\left(
1+\alpha_{2}(\alpha_{2}-1)\,\Delta_{\alpha_{2}}\right)  }} &
\mbox{when}\ \alpha_{2}<1\medskip\medskip\medskip\\
\displaystyle{\frac{1}{\alpha_{1}-1}}\cdot{\frac{\ln\left(  1+\alpha
_{1}(\alpha_{1}-1)\Delta_{\alpha_{1}}\right)  }{\Delta_{\alpha_{2}}}} &
\mbox{when}\ \alpha_{2}=1.\medskip
\end{array}
\right.
\end{multline}
The second assertion was for $\alpha_{1}=1$ deduced in Section~2 from the
lemmas presented there. The argument was based on the fact that $c_{\alpha
_{1}}(n)=1$ for $\alpha_{1}=1$. But $c_{\alpha}(n)=1$ for all $0<\alpha\leq1$
so that extension from $\alpha_{1}=1$\ to $0<\alpha_{1}<1$ is
straightforward.$\medskip$
\end{proof}

\begin{example}
\label{Example2}Let
\begin{equation}
P_{n}=\left(  p_{nj}\overset{def}{=}\frac{1_{\left\{  1\leq j\leq k/2\right\}
}}{\left\lfloor k/2\right\rfloor }\right)  ,\text{ \ \ }n=1,2,\ldots\label{57}%
\end{equation}
where $1_{A}$ is the indicator function, $\left\lfloor \cdot\right\rfloor $
stands for the integer part (floor function) and, as before,
\[
U=\left(  u_{j}\overset{def}{=}1/k:1\leq j\leq k\right)  .
\]
Then for$\ \alpha\neq0\medskip,1$
\begin{align*}
D_{\alpha}(P_{n},U)  &  ={\frac{\sum_{1}^{k}u_{j}\left(  \left(  p_{nj}%
/u_{j}\right)  ^{\alpha}-\alpha\left(  p_{nj}/u_{j}-1\right)  -1\right)
}{\alpha(\alpha-1)}}\\
&  ={\frac{\sum_{1}^{k}p_{nj}^{\alpha}u_{j}^{1-\alpha}+\sum_{1}^{k}%
(p_{nj}-u_{j})-\sum_{1}^{k}u_{j}}{\alpha(\alpha-1)}}\medskip\\
&  ={\frac{k^{\alpha-1}\sum_{1}^{\left\lfloor k/2\right\rfloor }\left\lfloor
k/2\right\rfloor ^{-\alpha}-1}{\alpha(\alpha-1)}}\hspace*{2.8cm}\medskip\\
&  ={\frac{k^{\alpha-1}\left\lfloor k/2\right\rfloor /\left\lfloor
k/2\right\rfloor ^{\alpha}-1}{\alpha(\alpha-1)}}\medskip\\
&  ={\frac{\left(  k/\left\lfloor k/2\right\rfloor \right)  ^{\alpha-1}%
-1}{\alpha(\alpha-1)}}.
\end{align*}
Therefore the identifiably condition (\ref{6}) takes on the form
\begin{multline*}
D_{\alpha}(P_{n},U)\longrightarrow\\
\left\{
\begin{array}
[c]{ll}%
\displaystyle{\frac{2^{\alpha-1}-1}{\alpha(\alpha-1)}}\overset{def}{=}%
\Delta_{\alpha} & ,\mbox{if}\ \alpha>0,\ \alpha\neq1\medskip\medskip\\
\displaystyle\ln2\ \overset{def}{=}\Delta_{1} & \mbox{if}\ \alpha=1.
\end{array}
\right.
\end{multline*}
If $0<\alpha\leq1$ then Lemma \ref{Lemma2} implies
\[
g_{\alpha}(\Delta)=\ln\left(  1+\alpha(\alpha-1)\,\Delta\right)  /(\alpha-1)
\]
when $0<\alpha<1$ and $g_{1}(\Delta)=\Delta$ when $\alpha=1$. If moreover
(\ref{68}) then under the alternative (\ref{57})
\begin{align*}
\frac{g_{\alpha}(\Delta_{\alpha})}{g_{1}(\Delta_{1})}  &  ={\frac{\ln\left(
1+\alpha(\alpha-1)\,\frac{2^{\alpha-1}-1}{\alpha(\alpha-1)}\right)  }%
{(\alpha-1)\ln2}}\medskip\\
&  ={\frac{\ln\left(  1+2^{\alpha-1}-1\right)  }{(\alpha-1)\ln2}}=1.
\end{align*}
Hence, by Definition 4, the likelihood ratio statistic $\widehat{D}_{1,n}$ is
as Bahadur efficient as any $\widehat{D}_{\alpha,n}$ with $0<\alpha<1$. If
$\alpha>1$ then Lemma \ref{Lemma2} implies
\[
{\frac{g_{\alpha}(\Delta_{\alpha})}{g_{1}(\Delta_{1})}}={\frac{(2^{\alpha
-1}-1)^{1/\alpha}}{\ln2}}>1.
\]
However, contrary to this prevalence of $g_{\alpha}(\Delta_{\alpha})\ $over
$g_{1}(\Delta_{1})$, Theorem~\ref{Theorem2} implies that $\widehat{D}_{1,n}$
is infinitely more Bahadur efficient than $\widehat{D}_{\alpha,n}$.$\medskip$
\end{example}

\begin{example}
\label{Example3}Let us now consider the truncated geometric distribution%
\[
P_{n}=(p_{n1},\ldots,p_{nk})=c_{k}(p)(1,p,\ldots,p^{k})
\]
with parameter $p=p_{n}\in]0,1[$. Since%
\[
1+p+p^{2}+\ldots={\frac{1}{1-p}}\text{ \ \ and \ \ }p^{k+1}+p^{k+2}%
+\ldots={\frac{p^{k+1}}{1-p},}%
\]
it holds%
\[
1+p+\cdots+p^{k}={\frac{1}{1-p}}-{\frac{p^{k+1}}{1-p}}={\frac{1-p^{k+1}}%
{1-p}=}\frac{1}{c_{k}(p)}.
\]
Hence for all $\alpha\neq0,1$%
\begin{align*}
\alpha(\alpha-1)\,D_{\alpha,n}+1 &  ={\frac{1}{k}}\sum_{j=0}^{k}\left(
{\frac{p_{nj}}{1/k}}\right)  ^{\alpha}\\
&  ={\frac{1}{k}}\sum_{j=1}^{k}{\frac{k^{\alpha}(1-p)^{\alpha}p^{\alpha j}%
}{(1-p^{k+1})^{\alpha}}}\medskip\\
&  ={\frac{\left(  k(1-p)\right)  ^{\alpha}}{k(1-p^{k+1})^{\alpha}}}\sum
_{j=0}^{k}(p^{\alpha})^{j}\medskip\\
&  ={\frac{\left(  k(1-p)\right)  ^{\alpha}}{k(1-p^{k+1})^{\alpha}}}%
\cdot{\frac{1-p^{\alpha(k+1)}}{1-p^{\alpha}}}\medskip\\
&  ={\frac{\left(  k(1-p)\right)  ^{\alpha}}{k(1-p^{\alpha})}}\cdot
{\frac{1-p^{\alpha(k+1)}}{(1-p^{k+1})^{\alpha}}}.
\end{align*}
In the particular case $p=1-x/k$\ for $x\neq0$\ fixed we get $k(1-p)=x$ and
\begin{align*}
k(1-p^{\alpha})\  &  =\ k\left(  1-\left(  1-{\frac{\alpha x}{k}}+o\left(
{\frac{x}{k}}\right)  \right)  \right)  \longrightarrow\alpha x,\medskip\\
p^{\alpha(k+1)} &  =\left(  1-{\frac{x}{k}}\right)  ^{\alpha(k+1)}%
\longrightarrow e^{-x\alpha},\medskip\\
p^{k+1} &  =\left(  1-{\frac{x}{k}}\right)  ^{k+1}\longrightarrow e^{-x}.
\end{align*}
Therefore
\begin{align*}
\alpha(\alpha-1)\,D_{\alpha,n}+1 &  ={\frac{x^{\alpha}}{k\left(  {\frac{\alpha
x}{k}}+o\left(  {\frac{x}{k}}\right)  \right)  }}\cdot{\frac{1-e^{-x\alpha}%
}{(1-e^{-x})^{\alpha}}}\medskip\\
&  ={\frac{x^{\alpha}}{\alpha x+o(x)}}\cdot{\frac{e^{x\alpha}-1}%
{(e^{x}-1)^{\alpha}}}.
\end{align*}
Consequently,
\[
\alpha(\alpha-1)\,\Delta_{\alpha}+1={\frac{x^{\alpha-1}}{\alpha}}\cdot
{\frac{e^{x\alpha}-1}{(e^{x}-1)^{\alpha}}}%
\]
i.e.,%
\[
\Delta_{\alpha}=\frac{x^{\alpha-1}(e^{x\alpha}-1)-\alpha(e^{x}-1)^{\alpha}%
}{\alpha^{2}(\alpha-1)(e^{x}-1)^{\alpha}}\text{ \ \ for }\alpha\neq0,1.
\]
By the L'Hospital rule,%
\begin{align*}
\Delta_{1}  & =\ln\frac{x}{e(e^{x}-1)}+\frac{xe^{x}}{e^{x}-1}~\text{,}\\
\Delta_{0}  & =\frac{\ln(e^{x}-1)-\ln x}{2}~.
\end{align*}
From here one can deduce that if $x\rightarrow0$ then
\[
\Delta_{\alpha}\longrightarrow0\text{ \ \ for all \ \ }\alpha\in\mathbb{R}.
\]
If $x=1$ then
\[
\Delta_{\alpha}={\frac{e^{\alpha}-1-(e-1)^{\alpha}}{\alpha^{2}(\alpha
-1).(e-1)^{\alpha}}}\text{ \ \ for }\alpha\neq0,1
\]
and%
\begin{align*}
\Delta_{1} &  =\frac{1-(e-1)\ln(e-1)}{e-1}=0.035,\ \ \ \ \\
\Delta_{0} &  =\frac{\ln(e-1)}{2}=0.271.
\end{align*}
Using Lemma 2 and Theorem 2 in a similar manner as in the previous example, we
find that here $\widehat{D}_{1,n}$ is more Bahadur efficient as any
$\widehat{D}_{\alpha,n}$ with $0<\alpha<\infty,$\ $\alpha\neq1.$
\end{example}

\section{Contiguity}

In this paper we proved that the statistics $\hat{D}_{\alpha,n}$ of orders
$\alpha>1$ are less Bahadur efficient than those of the orders $0<\alpha\leq1$
and that the latter are mutually comparable in the Bahadur sense. One may have
expected $\hat{D}_{1,n}$ to be much more Bahadur efficient than $\hat
{D}_{\alpha,n}$ for$\ 0<\alpha<1.$ In order to understand why this is not the
case we have to examine somewhat closer the assumptions of our theory.

Recall that given a sequence of pairs of probability measures $\left(
P_{n},Q_{n}\right)  _{n\in\mathbb{N}},$ $\left(  P_{n}\right)  _{n\in
\mathbb{N}}$ is said to be contiguous with respect to $\left(  Q_{n}\right)
_{n\in\mathbb{N}}$ if $Q_{n}\left(  A_{n}\right)  \rightarrow0$ for
$n\rightarrow\infty$ implies $P_{n}\left(  A_{n}\right)  \rightarrow0$ for
$n\rightarrow\infty$ and any sequence of sets $\left(  A_{n}\right)
_{n\in\mathbb{N}}.$ When $\left(  P_{n}\right)  _{n\in\mathbb{N}}$ is
contiguous with respect to $\left(  Q_{n}\right)  _{n\in\mathbb{N}}$ we write
$P_{n}\vartriangleleft Q_{n}.$ Let $P$ and $Q$ be probability measures on the
same set $\mathcal{X}$ and let $\left(  \mathcal{F}_{n}\right)  _{n\in
\mathbb{N}}$ be an increasing sequence of finite sub-$\sigma$-algebras on
$\mathcal{X}$ that generates the full $\sigma$-algebra on $\mathcal{X}.$ If
$P_{n}=P_{\mid\mathcal{F}_{n}}$ and $Q_{n}=Q_{\mid\mathcal{F}_{n}}$ then
$P_{n}\vartriangleleft Q_{n}$ if and only if $P\ll Q$ where $\ll$ denotes
absolute continuity. For completeness we give the proof of the following
simple proposition.$\medskip$

\begin{proposition}
Let $\left(  P_{n},Q_{n}\right)  _{n\in\mathbb{N}}$ denote a sequence of pairs
of probability measures and assume that the sequence $D_{1}\left(  P_{n}%
,Q_{n}\right)  $ is bounded. Then $P_{n}\vartriangleleft Q_{n}.\medskip$
\end{proposition}

\begin{proof}
Assume that the proposition is false. Then there exist $\varepsilon>0$ and a
subsequence of sets $\left(  A_{n_{k}}\right)  _{k\in\mathbb{N}}$ such that
$Q_{n_{k}}\left(  A_{n_{k}}\right)  \rightarrow0$ for $k\rightarrow\infty$ and
$P_{n_{k}}\left(  A_{n_{k}}\right)  \geq\varepsilon$ for all $k\in\mathbb{N}.$
\end{proof}

In general, a large power $\alpha$ makes the power divergence $D_{a}\left(
P,Q\right)  $ sensitive to large values of $dP/dQ.$ Therefore the statistics
$\widehat{D}_{\alpha,n}$ with large $\alpha$ should be used when the sequence
of alternatives $P_{n}$ may not be contiguous with respect to the sequence of
hypotheses $Q_{n}.$ Conversely, a small power $\alpha$ makes $D_{a}\left(
P,Q\right)  $ sensitive to small values of $dP/dQ.$ Therefore $\widehat
{D}_{\alpha,n}$ with small $\alpha$ should be used when the sequence of
hypotheses $Q_{n}$ is not contiguous with respect to the sequence alternatives
$P_{n}.$ Our conditions guarantee $P_{n}\vartriangleleft Q_{n}$ but not the
reversed contiguity $Q_{n}\vartriangleleft P_{n}.$ We see that a substantial
modification of the conditions is needed in order to guarantee that $\hat
{D}_{1,n}$ dominates the divergence statistcs $\hat{D}_{\alpha,n}$ of the
orders$\ 0<\alpha<1$ in the Bahadur sense$.$

\section{Appendix: Relations to previous results}

As mentioned at the end of Section II, Harremo\"{e}s and Vajda
\cite{Harremoes2008} assumed the same strong consistency as in Definition 4+
but introduced the Bahadur efficiency by the formula (\ref{19}). The next four
lemmas help to clarify the relation between this and the present precised
concept of Bahadur efficiency (\ref{18}).

Under the assumptions of Definition~4,\ \cite{Harremoes2008} considered the
following conditions.$\medskip$

\begin{description}
\item[\textbf{C1}:] The limit $\bar{c}_{\alpha_{2}/\alpha_{1}}$ considered in
(\ref{20}) exists.

\item[\textbf{C2:}] Both statistics $\widehat{D}_{\alpha_{i},n}$\ are strongly
consistent and both functions $g_{\alpha_{i}}$ are strongly Bahadur.$\medskip$
\end{description}

\begin{lemma}
\label{Lemma3}Let the assumptions of Definition~\ref{Def3} hold. Under
\textbf{C1} the Bahadur efficiency (\ref{19})\ coincides with the present
Bahadur efficiency (\ref{18}). If moreover \textbf{C2} holds then
(\ref{19})\ is the Bahadur efficiency in the strong sense.$\medskip$
\end{lemma}

\begin{proof}
The first assertion is clear from (\ref{19})\ and (\ref{18}). Under
\textbf{C2} the assumptions of Definition 3+ hold. Hence the second assertion
follows from Definition \ref{Def3+}.$\medskip$
\end{proof}

\begin{lemma}
\label{Lemma4}Let the\ assumptions of Definition 3 hold\ and let
$b(\alpha):\mathcal{I}\longrightarrow]0,1[$ be increasing and $d_{\alpha
}:\mathcal{I}\longrightarrow]0,\infty\lbrack$\ arbitrary function on an
interval $\mathcal{I}$\ covering $\{\alpha_{1},\alpha_{2}\}$. If for every
$\alpha\in\{\alpha_{1},\alpha_{2}\}$\ the sequence $c_{\alpha}(n)$\ generating
the Bahadur function $g_{\alpha}$ satisfies the asymptotic condition%
\begin{equation}
c_{\alpha}(n)=n^{b(\alpha)}(d_{\alpha}+o(1)) \label{21}%
\end{equation}
then (\ref{17}) holds for $c_{\alpha_{2}/\alpha_{1}}=\infty$ and condition
\textbf{C1} is satisfied.$\medskip$
\end{lemma}

\begin{proof}
Under (\ref{21}) it suffices to prove that (\ref{17}) holds for $c_{\alpha
_{2}/\alpha_{1}}=\infty,$ i.e.
\begin{equation}
\lim_{n\longrightarrow\infty}\frac{c_{\alpha_{2}}(m_{n})}{c_{\alpha_{1}}%
(n)}=\infty\label{22}%
\end{equation}
for $m_{n}$\ defined by (\ref{15}). By (\ref{21}),%
\[
c_{\alpha_{2}}(m_{n})=m_{n}^{b(\alpha_{2})}(d_{\alpha_{2}}+o(1))
\]
and
\[
c_{\alpha_{1}}(n)=n^{b(\alpha_{1})}(d_{\alpha_{1}}+o(1))
\]
so that (\ref{15}) implies
\[
m_{n}^{1-b(\alpha_{2})}=n^{1-b(\alpha_{1})}(\gamma\delta+o(1))
\]
for the finite positive constants
\[
\delta=\frac{d_{\alpha_{1}}}{d_{\alpha_{2}}}\text{ \ \ and \ \ }\gamma
={\frac{g_{\alpha_{1}}(\Delta_{\alpha_{1}})}{g_{\alpha_{2}}(\Delta_{\alpha
_{2}})}.}%
\]
Hence (\ref{16}) implies%
\begin{align*}
\frac{c_{\alpha_{2}}(m_{n})}{c_{\alpha_{1}}(n)}  & =\frac{m_{n}}{n}%
(\gamma^{-1}+o(1))\medskip\\
& =\frac{n^{\frac{1-b(\alpha_{1})}{1-b(\alpha_{2})}}}{n}((\gamma\delta
)^{\frac{1}{1-b(\alpha_{2})}}\gamma^{-1}+o(1))\medskip\\
& =n^{\frac{b(\alpha_{2})-b(\alpha_{1})}{1-b(\alpha_{2})}}(\gamma
^{\frac{b(\alpha_{2})}{1-b(\alpha_{2})}}\delta^{\frac{1}{1-b(\alpha_{2})}%
}+o(1))
\end{align*}
so that (\ref{22}) holds. $\medskip$
\end{proof}

\begin{lemma}
\label{Lemma5}Let the\ assumptions of Definition \ref{Def3}\ hold\ and let for
every $\alpha\in\{\alpha_{1},\alpha_{2}\}$\ the sequence $c_{\alpha}%
(n)$\ generating the Bahadur function $g_{\alpha}$ satisfy the asymptotic
condition%
\begin{equation}
c_{\alpha}(n)=\frac{\alpha n^{b(\alpha)}}{\ln n} \label{23}%
\end{equation}
for some increasing function $b(\alpha):\mathcal{I}\longrightarrow]0,1[$ on an
interval $\mathcal{I}$\ covering $\{\alpha_{1},\alpha_{2}\}$ . Then (\ref{17})
holds for $c_{\alpha_{2}/\alpha_{1}}=\infty$ and condition \textbf{C1} is
satisfied.$\medskip$
\end{lemma}

\begin{proof}
Similarly as before, it suffices to prove the relation (\ref{22}) for $m_{n}%
$\ defined by (\ref{15}). By (\ref{23}),%
\[
c_{\alpha_{2}}(m_{n})=\frac{\alpha_{2}m_{n}^{b(\alpha_{2})}}{\ln m_{n}}\text{
\ \ and \ \ }c_{\alpha_{1}}(n)=\frac{\alpha_{1}n^{b(\alpha_{1})}}{\ln n}%
\]
so that (\ref{15}) implies
\[
\frac{\alpha_{2}m_{n}^{1-b(\alpha_{2})}}{\ln m_{n}}=\frac{\alpha
_{1}n^{1-b(\alpha_{1})}}{\ln n}(\gamma+o(1))
\]
for the same $\gamma$\ as in the previous proof. Since $1-b(\alpha
_{2})<1-b(\alpha_{1})$, this implies the asymptotic relation
\begin{equation}
\frac{m_{n}}{n}\longrightarrow\infty{.} \label{24}%
\end{equation}
Similarly as in the previous proof, we get from (\ref{16})%
\begin{multline*}
\frac{c_{\alpha_{2}}(m_{n})}{c_{\alpha_{1}}(n)}=\frac{m_{n}}{n}(\gamma
^{-1}+o(1))\medskip=\\
\left(  \frac{\alpha_{1}\ln m_{n}}{\alpha_{2}\ln n}\right)  ^{\frac
{1}{1-b(\alpha_{2})}}n^{\frac{b(\alpha_{2})-b(\alpha_{1})}{1-b(\alpha_{2})}%
}\left(
\begin{array}
[c]{c}%
\gamma^{\frac{b(\alpha_{2})}{1-b(\alpha_{2})}}\\
+o(1)
\end{array}
\right)  \medskip\medskip\\
>n^{\frac{b(\alpha_{2})-b(\alpha_{1})}{1-b(\alpha_{2})}}(\gamma^{\frac
{b(\alpha_{2})}{1-b(\alpha_{2})}}+o(1)).
\end{multline*}
Therefore the desired relation (\ref{22}) holds.$\medskip$
\end{proof}

\begin{lemma}
\label{Lemma6}Let the\ assumptions of Definition 3\ hold\ and let for every
$\alpha\in\{\alpha_{1},\alpha_{2}\}$\ the sequence $c_{\alpha}(n)$\ generating
the Bahadur function $g_{\alpha}$ satisfy the asymptotic condition%
\begin{equation}
c_{\alpha}(n)=\frac{\alpha k^{b(\alpha)}}{\ln k} \label{25}%
\end{equation}
where $k=k_{n}\longrightarrow\infty$\ is the sequence considered above and
$b(\alpha):\mathcal{I}\longrightarrow]0,\infty\lbrack$ is increasing on an
interval $\mathcal{I}$\ covering$\{\alpha_{1},\alpha_{2}\}$ . Then (\ref{17})
holds for $c_{\alpha_{2}/\alpha_{1}}=\infty$ and condition \textbf{C1} is
satisfied.$\medskip$
\end{lemma}

\begin{proof}
It suffices to apply Lemma \ref{Lemma3} to the sequences
\[
c_{\alpha_{1}}(k)=\frac{\alpha_{1}k^{b(\alpha_{1})}}{\ln k}\text{ \ \ and
\ \ }c_{\alpha_{2}}(m_{k})=\frac{\alpha_{2}m_{k}^{b(\alpha_{2})}}{\ln m_{k}}%
\]
for $m_{k}$ defined by the condition%
\begin{equation}
\frac{m_{k}}{c_{\alpha_{2}}(m_{k})}=\frac{g_{\alpha_{1}}(\Delta_{\alpha_{1}}%
)}{g_{\alpha_{2}}(\Delta_{\alpha_{2}})}.\frac{k}{c_{\alpha_{1}}(k)}\left(
1+o(1)\right)  \text{ \ \ (cf. (\ref{15})).} \label{26}%
\end{equation}

\end{proof}

\begin{example}
\label{Example4}Let assumptions of Definition~\ref{Def3} hold for $\alpha
_{1}=1$\ and $\alpha_{2}=\alpha>1,$ and let
\begin{equation}
{\frac{k^{b(\alpha)+1}\ln n}{n}}\longrightarrow0\quad\mbox{for}\ \ \ b(\alpha
)=(\alpha-1)/\alpha. \label{27}%
\end{equation}
By \cite[Eq. 51, 76 and 79]{Harremoes2008} and (\ref{27}) the sequences%
\begin{equation}
c_{1}(n)=1\quad\mbox{and}\quad c_{\alpha}(n)={\frac{\alpha k^{b(\alpha)}}{\ln
k}} \label{28}%
\end{equation}
generate the Bahadur functions
\begin{equation}
g_{1}(\Delta)=\Delta\quad\mbox{and}\quad g_{\alpha}(\Delta)=\left(
\alpha(\alpha-1)\,\Delta\right)  ^{1/\alpha},\text{ \ \ }\Delta>0. \label{29}%
\end{equation}
Here we cannot apply Lemma \ref{Lemma4} since $c_{1}(n)$\ is not special case
of $c_{\alpha}(n)$\ for $\alpha=1.$\ An alternative direct approach can be
based on the observation that (\ref{15}) cannot hold if $\lim\inf_{n}%
m_{n}<\infty.$ In the opposite case $m_{n}\rightarrow\infty$ obviously implies%
\[
c_{\alpha/1}\overset{def}{=}\lim\nolimits_{n}{\frac{c_{\alpha}(m_{n})}%
{c_{1}(n)}}=\infty
\]
so that \textbf{C1} holds with $\bar{c}_{\alpha_{2}/\alpha_{1}}\equiv
c_{\alpha/1}=\infty$. Hence Lemma 1 implies that the Bahadur efficiency
$\mbox{BE}\left(  \widehat{D}_{1,n}\,;\,\widehat{D}_{\alpha,n}\right)
=\infty$ obtained previously by Harremo\"{e}s and Vajda \cite[Eq.
81]{Harremoes2008} coincides with the Bahadur efficiency of $\widehat{D}%
_{1,n}$\ with respect to $\widehat{D}_{\alpha,n}$\ in the present precised
sense of (\ref{18}). Under stronger condition on $k$\ than (\ref{27}),
Harremo\"{e}s and Vajda established also the strong consistency of the
statistics $\widehat{D}_{1,n}$\ and $\widehat{D}_{\alpha,n}$. One can verify
that (\ref{29}) are strongly Bahadur functions so that \textbf{C2} holds as
well. Hence, as argued by Lemma 3, we deal here with the Bahadur efficiency in
the strong sense.$\medskip$
\end{example}

\begin{example}
\label{Example5}Let assumptions of Definition~\ref{Def3} hold for $\alpha
_{1}>1$\ and let the function $b(\alpha)$\ be defined by (\ref{27}) for all
$\alpha\geq1$. Harremo\"{e}s and Vajda (2008) proved that if the sequence
$k$\ satisfies the condition (\ref{27}) with $\alpha=\alpha_{2}$\ then for all
$\alpha\in\{\alpha_{1},\alpha_{2}\}$\ the function $g_{\alpha}(\Delta)$ given
by the second formula in (\ref{29}) is Bahadur function of the statistics
$\widehat{D}_{\alpha,n}$ generated by the sequences $c_{\alpha}(n)$ from the
second formula in (\ref{28}). Thus in this case the assumptions of Lemma
\ref{Lemma4} hold. From Lemmas \ref{Lemma6} and \ref{Lemma3} we conclude that
the Bahadur efficiency
\[
\mbox{BE}\left(  \widehat{D}_{\alpha_{1},n}\,;\,\widehat{D}_{\alpha_{2}%
,n}\right)  =\infty\quad\mbox{for all}\ \text{\ \ }0<\alpha_{1}<\alpha
_{2}<\infty
\]
obtained in \cite[Eq. 81]{Harremoes2008} coincides with the Bahadur efficiency
in the present precise sense. Similarly as in the previous example, we can
arrive to the conclusion that this is the Bahadur efficiency in the strong
sense.$\medskip$

\textbf{Acknowledgement.} This research was supported by the European Network
of Excellence and by the GA\v{C}R grant 202/10/0618.
\end{example}

\nocite{Liese1987}\nocite{Liese2006}\nocite{Quine1985}\nocite{Morris1975}%
\nocite{Lehman2005}\nocite{Beirlant2001}\nocite{Csiszar1963}%
\nocite{Gyorfi2000}\nocite{Read1988}\nocite{Gyorfi2002}\nocite{Renyi1961}

\bibliographystyle{ieeetr}
\bibliography{database1}

\begin{thebibliography}{10}

\bibitem{Lehman2005}
E.~Lehman and G.~Castella, {\em Testing Statistical Hypotheses}.
\newblock New York: Springer, 3rd ed.~ed., 2005.

\bibitem{Csiszar1963}
I.~Csisz{\'a}r, ``Eine informationstheoretische {U}ngleichung und ihre
  {A}nwendung auf den {B}eweis der ergodizit{\"a}t von {M}arkoffschen
  {K}etten,'' {\em Publ. Math. Inst. Hungar. Acad.}, vol.~8, pp.~95--108, 1963.

\bibitem{Renyi1961}
A.~R{\'e}nyi, ``On measures of entropy and information,'' in {\em Proceedings
  of the Fourth Berkeley Symposium on Mathematical Statistics and Probability},
  vol.~1, pp.~547--561, 1961.

\bibitem{Kailath1967}
T.~Kailath, ``The divergence and {B}hattacharrya distance in signal
  selection.,'' {\em IEEE Trans. Comm.}, vol.~15, pp.~52--60, 1967.

\bibitem{Harremoes2008}
P.~Harremo{\"e}s and I.~Vajda, ``On the {B}ahadur-efficient testing of
  uniformity by means of the entropy,'' {\em IEEE Trans. Inform Theory},
  vol.~54, pp.~321--331, Jan. 2008.

\bibitem{Harremoes2008d}
P.~Harremo{\"e}s and I.~Vajda, ``Consistence of various $\phi$-divergence
  statistics,'' Research Report {\'U}TIA AV {\v C}R 2218, Institute of
  Information Theory and Automation, Praha, March 2008.

\bibitem{Harremoes2008e}
P.~Harremo{\"e}s and I.~Vajda, ``Efficiency of entropy testing,'' in {\em
  International Symposium on Information Theory}, pp.~2639--2643, IEEE, July
  2008.

\bibitem{Quine1985}
M.~P. Quine and J.~Robinson, ``Efficiencies of chi-square and likelihood ratio
  goodness-of-fit tests.,'' {\em Ann. Statist.}, vol.~13, pp.~727--742, 1985.

\bibitem{Beirlant2001}
J.~Beirlant, L.~Devroye, L.~Gy{\"o}rfi, and I.~Vajda, ``Large deviations of
  divergence measures on partitions,'' {\em J. Statist. Planning and Infer.},
  vol.~93, pp.~1--16, 2001.

\bibitem{Gyorfi2000}
L.~Gy{\"o}rfi, G.~Morvai, and I.~Vajda, ``Information-theoretic methods in
  testing the goodness-of-fit,'' in {\em Proc. International Symposium on
  Information Theory, Sorrento, Italy, June25-30}, p.~28, 2000.

\bibitem{Liese1987}
F.~Liese and I.~Vajda, {\em Convex Statistical Distances}.
\newblock Leipzig: Teubner, 1987.

\bibitem{Liese2006}
F.~Liese and I.~Vajda, ``On divergence and informations in statistics and
  information theory,'' {\em IEEE Tranns. Inform. Theory}, vol.~52, pp.~4394 --
  4412, Oct. 2006.

\bibitem{Read1988}
T.~R.~C. Read and N.~Cressie, {\em Goodness of Fit Statistics for Discrete
  Multivariate Data.}
\newblock Berlin: Springer, 1988.

\bibitem{Morris1975}
C.~Morris, ``Central limit theorems for multinomial sums.,'' {\em Ann.
  Statist.}, vol.~3, pp.~165--188, 1975.

\bibitem{Gyorfi2002}
L.~Gy{\"o}rfi and I.~Vajda, ``Asymptotic distributions for goodness-of-fit
  statistics in a sequence of multinomial models,'' {\em Stat. Probab.
  Letters}, vol.~56, no.~1, pp.~57--67, 2002.

\bibitem{Harremoes2001h}
P.~Harremo{\"e}s and F.~Tops{\o}e, ``Inequalities between entropy and index of
  coincidence derived from information diagrams,'' {\em IEEE Trans. Inform.
  Theory}, vol.~47, pp.~2944--2960, Nov. 2001.

\end{thebibliography}

\end{document}